\documentclass[11pt,twoside,a4paper]{amsart}

\usepackage{amssymb}
\usepackage{vmargin}
\usepackage{amscd}
\usepackage{stmaryrd}
\usepackage{mathrsfs}
\usepackage[all]{xy}

\usepackage{xr}
\usepackage{hyperref}
\hypersetup{colorlinks,
linkcolor=black,
citecolor=black,
urlcolor=myurlcolor}

\setmargins{32mm}{20mm}{14.6cm}{22cm}{1cm}{1cm}{1cm}{1cm}

\newcommand\da{\!\downarrow\!}

\newcommand\lra{\longrightarrow}

\newcommand\ten{\otimes}
\newcommand\hten{\hat{\otimes}}

\newcommand\eps{\epsilon}

\renewcommand\H{\mathrm{H}}
\newcommand\z{\mathrm{Z}}

\newcommand\HH{\mathrm{HH}}

\newcommand\Z{\mathbb{Z}}
\newcommand\Q{\mathbb{Q}}

\newcommand\Cx{\mathbb{C}}

\newcommand\bA{\mathbb{A}}

\newcommand\bG{\mathbb{G}}

\newcommand\bL{\mathbb{L}}

\newcommand\cA{\mathcal{A}}

\newcommand\cP{\mathcal{P}}

\newcommand\cS{\mathcal{S}}

\newcommand\cV{\mathcal{V}}

\renewcommand\O{\mathscr{O}}

\newcommand\sB{\mathscr{B}}

\newcommand\sD{\mathscr{D}}
\newcommand\sE{\mathscr{E}}

\newcommand\sL{\mathscr{L}}
\newcommand\sM{\mathscr{M}}

\newcommand\sO{\mathscr{O}}

\newcommand\sR{\mathscr{R}}

\newcommand\sT{\mathscr{T}}

\newcommand\fX{\mathfrak{X}}

\renewcommand\hom{\mathscr{H}\!\mathit{om}}
\newcommand\cHom{\mathcal{H}\!\mathit{om}}
\newcommand\cDiff{\mathcal{D}\!\mathit{iff}}

\newcommand\Alg{\mathrm{Alg}}
\newcommand\CAlg{\mathrm{CAlg}}

\newcommand\Hom{\mathrm{Hom}}

\newcommand\HHom{\underline{\mathrm{Hom}}}

\newcommand\Ext{\mathrm{Ext}}

\newcommand\Diff{\mathrm{Diff}}

\newcommand\im{\mathrm{Im\,}}

\newcommand\Ob{\mathrm{Ob}}
\newcommand\ob{\mathrm{ob}}

\newcommand\Co{\mathrm{Co}}
\newcommand\CoS{\mathrm{CoS}}

\newcommand\Spec{\mathrm{Spec}\,}

\newcommand\Set{\mathrm{Set}}

\newcommand\Aff{\mathrm{Aff}}

\newcommand\Sp{\mathrm{Sp}}
\newcommand\PreSp{\mathrm{PreSp}}

\newcommand\Pol{\mathrm{Pol}}

\newcommand\Comp{\mathrm{Comp}}

\newcommand\nondeg{\mathrm{nondeg}}

\newcommand\ad{\mathrm{ad}}

\renewcommand\>{\rangle}
\newcommand\Lim{\varprojlim}
\newcommand\LLim{\varinjlim}
\DeclareMathOperator*{\holim}{holim}
\newcommand\ho{\mathrm{ho}\!}

\newcommand\onto{\twoheadrightarrow}

\newcommand\xra{\xrightarrow}

\newcommand\pr{\mathrm{pr}}

\newcommand\bt{\bullet}
\newcommand\by{\times}

\newcommand\mc{\mathrm{MC}}
\newcommand\mmc{\underline{\mathrm{MC}}}

\newcommand\Rees{\mathrm{Rees}}
\newcommand\Symm{\mathrm{Symm}}

\newcommand\et{\acute{\mathrm{e}}\mathrm{t}}

\newcommand\Tot{\mathrm{Tot}\,}

\newcommand\pd{\partial}

\newcommand\half{\frac{1}{2}}

\newcommand\gr{\mathrm{gr}}

\newcommand\Fil{\mathrm{Fil}}

\newcommand\red{\mathrm{red}}

\newcommand\Lie{\mathrm{Lie}}

\newcommand\dR{\mathrm{dR}}
\newcommand\DR{\mathrm{DR}}

\newcommand\op{\mathrm{opp}}

\newcommand\co{\colon\thinspace}

\newcommand\oR{\mathbf{R}}

\newcommand\oL{\mathbf{L}}

\newcommand\uleft\underleftarrow
\newcommand\uline\underline
\newcommand\uright\underrightarrow

\newtheorem{theorem}{Theorem}[section]
\newtheorem{proposition}[theorem]{Proposition}

\newtheorem{lemma}[theorem]{Lemma}
\newtheorem*{theorem*}{Theorem}
\newtheorem*{proposition*}{Proposition}
\newtheorem*{corollary*}{Corollary}
\newtheorem*{lemma*}{Lemma}
\newtheorem*{conjecture*}{Conjecture}

\theoremstyle{definition}
\newtheorem{definition}[theorem]{Definition}
\newtheorem*{definition*}{Definition}

\newtheorem*{notation*}{Notation}

\theoremstyle{remark}
\newtheorem{example}[theorem]{Example}

\newtheorem{remark}[theorem]{Remark}
\newtheorem{remarks}[theorem]{Remarks}

\newtheorem{properties}[theorem]{Properties}

\newtheorem*{example*}{Example}
\newtheorem*{examples*}{Examples}
\newtheorem*{remark*}{Remark}
\newtheorem*{remarks*}{Remarks}
\newtheorem*{exercise*}{Exercise}
\newtheorem*{property*}{Property}
\newtheorem*{properties*}{Properties}

\begin{document}

\begin{abstract}
We formulate a notion of $E_0$ quantisation of $(-1)$-Poisson structures on derived Artin $N$-stacks, and
construct a map from   $E_0$ quantisations of $(-1)$-shifted symplectic structures to power series in  de Rham cohomology. For a square root of the dualising line bundle, this gives an equivalence between even power series and  self-dual  quantisations. In particular, there is a canonical quantisation of any such square root, which  localises to recover the perverse sheaf of vanishing cycles on derived DM stacks, thus giving a form of derived  categorification of Donaldson--Thomas invariants.
\end{abstract}

\title[Quantisation for $(-1)$-shifted symplectic structures]{Deformation quantisation for $(-1)$-shifted symplectic structures and vanishing cycles}
\author{J.P.Pridham}
\thanks{This work was supported by  the Engineering and Physical Sciences Research Council [grant number EP/I004130/2].}


\maketitle

\section*{Introduction}

In \cite{PTVV}, the notion of shifted symplectic structures on derived stacks was introduced. Examples of derived stacks carrying $(-1)$-shifted symplectic structures are Lagrangian intersections  and character varieties of threefolds. In \cite{BBDJS}, a perverse sheaf of vanishing cycles was then constructed on any oriented $(-1)$-shifted symplectic derived $\Cx$-scheme, giving a categorification of Behrend's function \cite{behrendDTmicrolocal}, and hence of Donaldson--Thomas invariants when applied to the derived moduli stack  of sheaves on a Calabi--Yau threefold. 

As discussed in \cite{toenICM}, there is a general philosophy that $n$-shifted symplectic structures should give rise to $E_{n+1}$-algebra quantisations of the structure sheaf. For $n=-1$, this means some sort of deformation as an abelian group, and there was an expectation that this should be related to the perverse sheaf  $\cP\cV$ of \cite{BBDJS} and to BV algebras as in \cite[\S 3.4.3]{behrendICM}. In this paper, we realise these expectations by formulating and establishing $E_0$-quantisation as a genuinely derived object for each oriented $(-1)$-shifted symplectic structure, and show that its localisation  recovers $\cP\cV$.

A $(-1)$-shifted Poisson structure on a derived scheme $X$ is a Maurer--Cartan element $\pi= \sum_{i\ge 2} \pi_i$ with $\pi_i \in \Symm^i \sT_X$. For a line bundle $\sL$, we define an $E_0$ quantisation of $(\pi, \sL)$ to be a Maurer--Cartan element $\Delta = \sum_{i\ge 2} \Delta_i \hbar^{i-1} \in \sD_X(\sL)\llbracket\hbar\rrbracket$ such that $\Delta_i$ is a differential operator of order $i$ lifting $\pi_i$. 

The proof in \cite{poisson} of the correspondence between $n$-shifted symplectic and non-degenerate Poisson structures relied on the existence, for all Poisson structures $\pi$, of a map $\mu(-,\pi)$ from the de Rham algebra to the algebra $T_{\pi}\widehat{\Pol}(X,n)$  of shifted polyvectors with differential twisted by $\pi$. Since $[\pi,-]$ defines a derivation from $\sO_X$ to  $T_{\pi}\widehat{\Pol}(X,n)$, it determines a map $\Omega^1_X \to T_{\pi}\widehat{\Pol}(X,n)[1]$,  and    $\mu(-,\pi)$ is the resulting morphism of CDGAs.

We adapt this idea to construct (Lemma  \ref{QPolmudef}), for any $E_0$ quantisation $\Delta$ of a CDGA $A$, an $A_{\infty}$-morphism $\mu(-,\Delta)$ from the de Rham algebra $\DR(A)$ to $\sD_A\llbracket\hbar\rrbracket$. Roughly speaking, this is an $A$-algebra homomorphism, with the restriction to  $\Omega^1_A$ corresponding to the derivation $[\Delta, -]$; because $\sD_A$ is not commutative, we have to define $\mu$ explicitly on an associative algebra resolution of the de Rham algebra.

This gives rise to a notion of compatibility between $E_0$ quantisations $\Delta$ and generalised $(-1)$-shifted pre-symplectic structures (power series $\omega$ of elements of the de Rham complex): we say that $\omega$ and $\Delta$ are compatible if
\[
 \mu(\omega, \Delta) \simeq \hbar^2  \frac{\pd \Delta}{\pd \hbar}. 
\]
Proposition \ref{QcompatP1} shows that every non-degenerate affine quantisation $\Delta$ has a unique compatible generalised pre-symplectic structure, thus giving us a map
\[
 Q\cP(A,-1)^{\nondeg} \to \H^1(F^2\DR(A)) \by \hbar\H^1(\DR(A))\llbracket \hbar\rrbracket
\]
on the space of non-degenerate $E_0$ quantisations.

In fact, much more is true. We have spaces $Q\cP(A,-1)/G^{k+1}$ consisting of $E_0$ quantisations of order $k$, 
by which we mean Maurer--Cartan elements in $\prod_{j \ge 2} (F_j\sD_A/F_{j-k-1}) \hbar^{j-1}$, for $F$  the order filtration on $\sD$. We then have maps
\[
 Q\cP(A,-1)^{\nondeg}/G^{k+1} \to \H^1(F^2\DR(A)) \by \hbar \H^1(\DR(A)) [\hbar]/\hbar^k,
\]
and Proposition \ref{quantprop} shows  that the resulting map
\[
 Q\cP(A,-1)^{\nondeg} \to (Q\cP(A,-1)^{\nondeg}/G^2) \by \hbar^2\H^1(\DR(A))\llbracket \hbar\rrbracket
\]
underlies an equivalence. Thus  quantisation reduces to a first order problem.

Section \ref{DMsn} interposes some abstract nonsense to transfer these results from affine derived schemes to derived DM $N$-stacks (Propositions \ref{QcompatP2} and \ref{quantprop2}). Section \ref{Artinsn} extends the results of \S \ref{affinesn} to the formalism of commutative bidifferential bigraded algebras,
and thus to derived Artin $N$-stacks (Proposition \ref{prop3}).

 When $\sL$ is Grothendieck--Verdier self-dual (i.e. a square root of the dualising line bundle $\omega_X$), or more generally whenever $\sD(\sL) \simeq \sD(\sL)^{\op}$, \S \ref{sdsn} introduces a notion of self-duality for  quantisations $\Delta$ of $\sL$. For self-dual quantisations, the first order obstruction vanishes, and in fact the equivalence class of such quantisations of a non-degenerate $(-1)$-shifted  Poisson structure is canonically isomorphic to 
\[
 \hbar^2\H^1_{\dR}(X) \llbracket \hbar^2\rrbracket.
\]
In particular, there is an $\infty$-functor from the space of $(-1)$-shifted symplectic structures to deformation quantisations of $\sL$; for the symplectic structure on a derived critical locus, this quantisation is just given by a twisted Hodge complex, so (Proposition \ref{PVprop}) localising at $\hbar$ recovers the perverse sheaf of vanishing cycles studied in \cite{BBDJS}.

In \S \ref{nonnegsn}, we discuss how to adapt these results  to quantisation of $n$-shifted symplectic structures for $n \ge -2$. For $n=-2$, quantisations should arise as Maurer--Cartan elements of a $BV$-algebra quantisation of $-2$-shifted polyvectors. For $n\ge 0$, formality of the  $E_{n+2}$ operad  allows the construction of a compatibility map $\mu$ from de Rham cohomology to quantised Poisson cohomology,  leading to a map from non-degenerate $E_{n+1}$-quantisations of $\sO_X$ to 
power series 
\[
 \H^{n+2}(F^2\DR(A)) \by \hbar\H^{n+2}(\DR(A))\llbracket \hbar\rrbracket.
\]
 Again, the only obstruction to deforming a non-degenerate Poisson structure is first order, and we sketch a notion of self-duality under which the obstruction vanishes --- for details, see \cite{DQnonneg}. 

\tableofcontents

\section{Compatible quantisations on derived affine schemes}\label{affinesn}

In this section, we develop the notion of compatibility between $E_0$ quantisations and generalised $(-1)$-shifted pre-symplectic structures on derived affine schemes, ultimately reducing deformation quantisation to a first-order problem.

Let $R$ be a graded-commutative differential algebra (CDGA) over $\Q$, and fix a CDGA $A$ over $R$. We will denote the differentials on $A$ and $R$  by $\delta$. 

Given $R$-modules $M,N$ in cochain complexes, we write $\HHom_R(M,N)$ for the cochain complex given by
\[
  \HHom_R(M,N)^i= \Hom_{R^{\#}}(M^{\#},N^{\#[i]}),
\]
 with differential $\delta f= \delta_N \circ f \pm f \circ \delta_M$,
 where $V^{\#}$ denotes the graded module underlying a cochain complex $V$.

\subsection{Differential operators and quantised polyvectors}

We now formulate a theory of differential operators for our CDGA $A$, leading to a notion of deformation quantisation for a $(-1)$-shifted Poisson structure.

\subsubsection{Differential operators}

\begin{definition}
Given  $A$-modules $M,N$ in cochain complexes, inductively define the filtered cochain complex $\Diff(M,N)= \Diff_{A/R}(M,N)\subset \HHom_R(M,N)$ of differential operators from $M$ to $N$  by setting
\begin{enumerate}
 \item $F_0 \Diff(M,N)= \HHom_A(M,N)$,
\item $F_{k+1} \Diff(M,N)=\{ u \in \HHom_R(M,N)~:~  [a,u]\in F_{k} \Diff(M,N)\, \forall a \in A \}$, where $[a,u]= au- (-1)^{\deg a\deg u} ua$.
\item $\Diff(M,N)= \LLim_k F_k\Diff(M,N)$.
\end{enumerate}
(The reason for the notation $F$ is that $F^p:= F_{-p}$ gives a Hodge filtration.)
\end{definition}

The definitions ensure that the associated gradeds $\gr^F_k\Diff_A(M,N)$ have the structure of $A$-modules. Also note that for any $u \in F_{k+1} \Diff(M,N)$ the commutator $[u,-]$ defines a derivation from $A$ to $\gr^F_k \Diff(M,N)$, giving an $A$-linear map
\[
 \gr^F_{k+1} \Diff(M,N) \to \HHom_A(\Omega^1_A,\gr^F_k \Diff(M,N)). 
\]
Proceeding inductively and considering the symmetries involved, this gives  maps
\[
 \gr^F_{k} \Diff(M,N) \to \HHom_A(M\ten_A\CoS^k_A\Omega^1_A,N)
\]
for all $k$. [Here,  $\CoS_A^p(M) =\Co\Symm^p_A(M)= (M^{\ten_A p})^{\Sigma_p}$ and $\Co\Symm_A(M) = \bigoplus_{p \ge 0}\CoS_A^p(M)$.]

 These maps will be isomorphisms whenever $A$ is semi-smooth in the sense that the underlying  graded algebra $A^{\#}$  is isomorphic to $(R^{\#}\ten_{R^0}S)[P^{\#}]$, for $S$ a smooth $R^0$-algebra and $P^{\#}$   a graded projective module over $R^{\#}\ten_{R^0}S$ (in particular, if $A$ is cofibrant as a CDGA over $R$), and $M^{\#}$ is projective over $A^{\#}$.

Also observe that for $A$-modules $M,N,P$, the composition map $\HHom_R(N,P)\ten_R\HHom_R(M,N)\xra{\circ} \HHom_R(M,P)$ restricts to give $F_l\Diff_{A/R}(N,P)\ten_RF_k\Diff_{A/R}(M,N) \to F_{k+l}\Diff_{A/R}(M,P)$.

\begin{definition}
Given  an $A$-module $M$ in cochain complexes, write $\sD(M)= \sD_{A/R}(M):= \Diff_{A/R}(M,M)$,   which we regard as a  DGAA under the composition above. We simply write $\sD_A= \sD_{A/R}$ for  $\sD_{A/R}(A,A)$. 
\end{definition}

For $A$-modules $M,N$, inclusion in $\Hom_R(M,N)$ gives a natural map $\HHom_A(M, (N\ten_A\sD_A)^r) \to   \Diff(M,N) $, where $(-)^r$ denotes the right $A$-module structure.  This will be an isomorphism whenever $A$ is semi-smooth over $R$  (in particular, if $A $ is cofibrant as a CDGA over $R$). 

\begin{remark}\label{alternateDAremark}
For $A$ semi-smooth over $R$, there is an equivalent alternative description of $\Diff(M,N)$, familiar from the underived setting. The algebra $A\ten_RA$ is naturally an $A$-bimodule, and if we write $I$ for the diagonal ideal $\ker(A\ten_RA \to A)$, then there are natural isomorphisms
\[
\alpha \co \HHom_{A}( (A\ten_RM)/ \overbrace{I \cdots I}^{k+1}M, N) \to F_k\Diff(M,N),
\]
given by $\alpha(f)(m):= f(1\ten m)$.
\end{remark}

\subsubsection{Polyvectors}

The following is adapted from \cite[Definition \ref{poisson-poldef}]{poisson}, with the introduction of a dummy variable $\hbar$ of cohomological degree $0$.

\begin{definition}\label{poldef}
Define the complex of $(-1)$-shifted polyvector fields  on $A$ by
\[
 \widehat{\Pol}(A/R,-1):= \prod_{p \ge 0}\HHom_A(\CoS_A^p(\Omega^1_{A/R}),A)\hbar^{p-1}. 
\]
with graded-commutative  multiplication $(a,b)\mapsto \hbar ab$  following the usual conventions for symmetric powers.  

The Lie bracket on $\Hom_A(\Omega^1_{A/R},A)$ then extends to give a bracket (the Schouten--Nijenhuis bracket)
\[
[-,-] \co \widehat{\Pol}(A/R,-1)\by \widehat{\Pol}(A/R,-1)\to \widehat{\Pol}(A/R,-1),
\]
determined by the property that it is a bi-derivation with respect to the multiplication operation. 

Thus  $\widehat{\Pol}(A/R,-1)$ has the natural structure of a $P_1$-algebra  (i.e. a Poisson algebra), and in particular $\widehat{\Pol}(A/R,-1)$ is a differential graded Lie algebra (DGLA) over $R$.

Note that the differential $\delta$ on $\widehat{\Pol}(A/R,-1)$ can be written as $[\delta,-]$, where $\delta \in \widehat{\Pol}(A/R,-1)^1$ is the element defined by the derivation $\delta$ on $A$.
\end{definition}

Strictly speaking, $\widehat{\Pol}$ is the complex of multiderivations, as polyvectors are usually defined as symmetric powers of the tangent complex. The two definitions agree (modulo completion) whenever the tangent complex is perfect, and Definition \ref{poldef} is the more natural object when the definitions differ. 

\begin{definition}\label{Fdef}
Define a decreasing filtration $F$ on  $\widehat{\Pol}(A/R,-1)$ by 
\[
 F^i\widehat{\Pol}(A/R,-1):= \prod_{j \ge i}\HHom_A( \CoS_A^j\Omega^1_{A/R},A)\hbar^{j-1};
\]
this has the properties that $\widehat{\Pol}(A/R,-1)= \Lim_i \widehat{\Pol}(A/R,-1)/F^i$, with $[F^i,F^j] \subset F^{i+j-1}$, $\delta F^i \subset F^i$, and $ F^i F^j \subset \hbar^{-1} F^{i+j}$.
\end{definition}

Observe that this filtration makes $F^2\widehat{\Pol}(A/R,-1)$ into a pro-nilpotent DGLA.

\begin{definition}\label{Tpoldef0}
 Define the tangent DGLA of polyvectors by 
\[
 T\widehat{\Pol}(A/R,-1):= \widehat{\Pol}(A/R,-1)\oplus \prod_{p \ge 0}\HHom_A(\CoS_A^p(\Omega^1_{A/R}),A)\hbar^{p}\eps,
\]
for $\eps$ of degree $0$ with $\eps^2=0$. The Lie bracket is given by $ [u+v\eps, x+y\eps]= [u,x]+ [u,y]\eps + [v,x]\eps$.
\end{definition}

\begin{definition}\label{Tpoldef}
Given a Maurer--Cartan element $\pi \in  \mc(F^2\widehat{\Pol}(A/R,-1)) $, define 
\[
 T_{\pi}\widehat{\Pol}(A/R,-1):= \prod_{p \ge 0}\HHom_A(\CoS_A^p(\Omega^1_{A/R}),A)\hbar^{p},
\]
with derivation $\delta + [\pi,-]$ (necessarily square-zero by the Maurer--Cartan conditions). 

The product on polyvectors makes this a  CDGA (with no need to rescale the product by $\hbar$), and it inherits the filtration $F$ from $\widehat{\Pol}$. 

Given $\pi \in \mc(F^2\widehat{\Pol}(A/R,-1)/F^p)$, we define $T_{\pi}\widehat{\Pol}(A/R,-1)/F^p$ similarly. This is a CDGA  because $F^i\cdot F^j \subset F^{i+j}$. 
\end{definition}

Regarding $ T_{\pi}\widehat{\Pol}(A/R,-1)$ as an abelian DGLA, observe that $\mc(T_{\pi}\widehat{\Pol}(A/R,-1))$ is just the fibre of $\mc( T\widehat{\Pol}(A/R,-1)) \to  \mc(\widehat{\Pol}(A/R,-1))$ over $\pi$. Evaluation at $\hbar=1$ gives an  isomorphism from $T\widehat{\Pol}(A/R,-1)$ to the DGLA   $\widehat{\Pol}(A/R,-1)\ten_{\Q} \Q[\eps]$ of \cite[\S \ref{poisson-polsn}]{poisson}, and the map $\sigma$ of \cite[Definition \ref{poisson-sigmadef}]{poisson} then becomes:

\begin{definition}\label{sigmadef}
Define 
\[
\sigma =- \pd_{\hbar^{-1}} \co \widehat{\Pol}(A/R,-1) \to T\widehat{\Pol}(A/R,-1)
\]
 by $\alpha \mapsto \alpha + \eps \hbar^{2}\frac{\pd \alpha}{\pd \hbar}$. Note that this is a morphism of filtered  DGLAs, so gives a map $\mc(F^2\widehat{\Pol}(A/R,-1)) \to \mc(F^2T\widehat{\Pol}(A/R,-1))$, with $ \sigma(\pi) \in \z^1(F^2T_{\pi}\widehat{\Pol}(A/R,-1))$. 
\end{definition}

\subsubsection{Quantised $(-1)$-shifted polyvectors}\label{qpolsn}

\begin{definition}\label{strictlb}
 Define a strict line bundle over $A$ to be an $A$-module $M$ in cochain complexes such that $M^{\#}$ is a projective module of rank $1$ over the graded-commutative algebra $A^{\#}$ underlying $A$.  Given $b \in \z^1A$, define $A_b$ to be the strict line bundle $(A, \delta +b)$.
\end{definition}
(When $A$ has elements of positive degree, note that $M$ might not be cofibrant in the projective model structure, but this does not affect anything.)

\begin{definition}\label{qpoldef}
Given a strict line bundle $M$ over $A$, define the complex of quantised $(-1)$-shifted polyvector fields  on $M$ by
\[
 Q\widehat{\Pol}(M,-1):= \prod_{p \ge 0}F_p\sD_{A/R}(M)\hbar^{p-1}. 
\]
Multiplication of differential operators gives us a product
\[
 Q\widehat{\Pol}(M,-1) \by Q\widehat{\Pol}(M,-1) \to \hbar^{-1} Q\widehat{\Pol}(M,-1),
\]
but the associated commutator $[-,-]$ takes values in $Q\widehat{\Pol}(M,-1)$, so $Q\widehat{\Pol}(M,-1)$ is a DGLA over $R$.
\end{definition}

Note that the differential $\delta$ on $Q\widehat{\Pol}(M,-1)$ can be written as $[\delta_M,-]$, where $\delta_M \in F_1\sD_{A/R}(M)^1$ is the element defined by the differential $\delta_M$ on $M$. In particular,  $Q\widehat{\Pol}(A_b,-1)$ is the graded associative algebra $Q\widehat{\Pol}(A,-1)$ with differential $[\delta_A +b , -]$.

\begin{definition}\label{QFdef}
Define a decreasing filtration $\tilde{F}$ on  $Q\widehat{\Pol}(M,-1)$ by 
\[
 \tilde{F}^iQ\widehat{\Pol}(M,-1):= \prod_{j \ge i} F_j\sD_{A/R}(M)\hbar^{j-1};
\]
this has the properties that $Q\widehat{\Pol}(M,-1)= \Lim_i \widehat{\Pol}(M,-1)/\tilde{F}^i$, with $[\tilde{F}^i,\tilde{F}^j] \subset \tilde{F}^{i+j-1}$, $\delta \tilde{F}^i \subset \tilde{F}^i$, and $ \tilde{F}^i \tilde{F}^j \subset \hbar^{-1} \tilde{F}^{i+j}$.
\end{definition}

Observe that this filtration makes $\tilde{F}^2Q\widehat{\Pol}(M,-1)$ into a pro-nilpotent DGLA.

\begin{definition}
When $A$ is cofibrant, we define an $E_0$ quantisation of $M$ over $R$ to be a Maurer--Cartan element
\[
 \Delta \in \mc(\tilde{F}^2Q\widehat{\Pol}(M,-1)).
\]
The associated $R\llbracket\hbar\rrbracket$-module $M_{\Delta}$ is given by $M\llbracket\hbar\rrbracket$ equipped with the differential $\delta_M +\Delta$, which is necessarily square-zero by the Maurer--Cartan condition. We then have $M_{\Delta}/(\hbar M_{\Delta})=M$, because $\hbar\mid \Delta$.  
\end{definition}

\begin{remark}\label{DRrmks}
 A more conceptual way to interpret such $E_0$ quantisations is as deformations of $M$ as a module over  the de Rham pro-algebra $\DR(A/R)$. Such a deformation of $M$  is the same as a deformation of the right $\sD_A$-module $M\ten_A\sD_A$, and a  Maurer--Cartan element $\Delta$ gives a deformation $x \mapsto \delta(x) +  \Delta\cdot x$
of the differential $\delta$ on $M\ten_A\sD_A\llbracket\hbar\rrbracket$. 

Equivalently, a $(-1)$-shifted Poisson structure on $\sO_X$ is the structure of an $\sO_X\llbracket\Omega^1_X[-1]\rrbracket= \hat{\O}_{TX[1]}$-module, and an $E_0$ quantisation is a lifting of such a structure making $\sO_X\llbracket\hbar\rrbracket$ a module over the Rees pro-algebra $\sR:=\prod_{p\in \Z} \hbar^{-p} F^p\DR(A/R)$, via the isomorphism $\sR/\hbar\sR\cong \hat{\O}_{TX[1]}$.

 In many ways, deformations over $\DR(A/R)$ are more natural than $R$-module deformations, because the de Rham algebra is the natural algebraic analogue of the analytic sheaf of complex constants. For $n>0$, $E_n$-algebra deformations over $R$ and over $\DR(A/R)$ are the same by the HKR isomorphism, but for $n=0$ the $\DR(A/R)$-module structure  imposes the condition that deformations be given by differential operators. 
\end{remark}

\begin{remark}\label{BVrmk}
Observe that when the $E_0$ quantisation $\Delta$ is linear in $\hbar$, it is a second-order differential operator. When $M=A$ and $\Delta(1)=0$, this gives exactly the structure of a $BV$-algebra over $R\llbracket\hbar\rrbracket$, the associated Lie bracket being given by the image of $\Delta$ in $\hbar(F_2\sD_A/F_1\sD_A)\cong \hbar\Symm^2_A\sT_A$. In general,  for  any $E_0$ quantisation $\Delta$ of $A$ with  $\Delta(1)=0$, the pair $(A, \Delta)$ is a commutative $BV_{\infty}$-algebra in the sense of \cite[Definition 9]{kravchenko}.
\end{remark}

\begin{remark}
 For unbounded CDGAs, the hypothesis that $A$ be cofibrant  seems unnecessarily strong. Most of our results will hold when $A^{\#}$ is free or even when $A$ is  semi-smooth. This suggests that the most natural notion of equivalence for CDGAs in this setting might not be quasi-isomorphism, but Morita equivalence of derived categories of the second kind (in the sense of \cite{positselskiDerivedCategories}). Dealing with semi-smooth CDGAs might provide an alternative approach to the stacky CDGAs featuring in \S \ref{Artinsn} to study Artin stacks.
\end{remark}

\begin{definition}\label{mcPLdef}
 Given a   DGLA $L$, define the the Maurer--Cartan set by 
\[
\mc(L):= \{\omega \in  L^{1}\ \,|\, d\omega + \half[\omega,\omega]=0 \in   L^{2}\}.
\]

Following \cite{hinstack}, define the Maurer--Cartan space $\mmc(L)$ (a simplicial set) of a nilpotent  DGLA $L$ by
\[
 \mmc(L)_n:= \mc(L\ten_{\Q} \Omega^{\bt}(\Delta^n)),
\]
where 
\[
\Omega^{\bt}(\Delta^n)=\Q[t_0, t_1, \ldots, t_n,\delta t_0, \delta t_1, \ldots, \delta t_n ]/(\sum t_i -1, \sum \delta t_i)
\]
is the commutative dg algebra of de Rham polynomial forms on the $n$-simplex, with the $t_i$ of degree $0$.
\end{definition}
%

\begin{definition}\label{Gdef}
We now define another decreasing filtration $G$ on  $Q\widehat{\Pol}(M,-1)$ by setting
\[
 G^iQ\widehat{\Pol}(M,-1):= \hbar^{i}Q\widehat{\Pol}(M,-1).
\]
We then set $G^i \tilde{F}^p:= G^i \cap \tilde{F}^p$.
\end{definition}

 Note that $G^i \subset \tilde{F}^i$, and  beware that $G^i \tilde{F}^p$ is not the same as $\hbar^{i} \tilde{F}^p$ in general, since 
\begin{align*}
 G^i\tilde{F}^pQ\widehat{\Pol}(M,-1) &= \prod_{j \ge p} F_{j-i}\sD_{A/R}(M)\hbar^{j-1}\\
\hbar^{i}\tilde{F}^pQ\widehat{\Pol}(M,-1) &= \prod_{j \ge p+i} F_{j-i}\sD_{A/R}(M)\hbar^{j-1}.
\end{align*}

We will also consider the convolution  $G*\tilde{F}$, given by $(G*\tilde{F})^p:= \sum_{i+j=p}G^i\cap\tilde{F}^j$ ; explicitly,
\[
(G*\tilde{F})^p Q\widehat{\Pol}(M,-1) = 
\prod_{j<p}F_{2j-p}\sD_{A/R}(M)\hbar^{j-1} \by \prod_{j \ge p} F_j\sD_{A/R}(M)\hbar^{j-1}.
\]
In particular, $(G*\tilde{F})^2 Q\widehat{\Pol}(M,-1) = A \oplus \tilde{F}^2 Q\widehat{\Pol}(M,-1)$.


\begin{definition}\label{Qpoissdef}
Define  the space $Q\cP(M,-1)$ of $E_0$ quantisations of $M$ over $R$  to be given by the simplicial 
set
\[
 Q\cP(M,-1):= \Lim_i \mmc( \tilde{F}^2 Q\widehat{\Pol}(M,-1)/\tilde{F}^{i+2}).
\]
Also write
\[
 Q\cP(M,-1)/G^k:= \Lim_i\mmc(\tilde{F}^2 Q\widehat{\Pol}(M,-1)/(\tilde{F}^{i+2}+G^k)),
\]
 so $Q\cP(M,-1)= \Lim_k Q\cP(M,-1)/G^k$. 

We will also consider twisted quantisations
\[
 Q^{tw}\cP(M,-1):= \Lim_i \mmc( (G*\tilde{F})^2 Q\widehat{\Pol}(M,-1)/\tilde{F}^{i+2});
\]
these are just quantisations of strict line bundles $M\ten_AA_b$ for $b \in \z^1(A)$.
\end{definition}

\subsubsection{The centre of a quantisation}\label{centresn}

\begin{definition}\label{TQpoldef0}
 Define the filtered tangent DGLA of quantised polyvectors by 
\begin{align*}
 TQ\widehat{\Pol}(M,-1)&:= Q\widehat{\Pol}(M,-1)\oplus \prod_{p \ge 0}F_p\sD_{A/R}(M)\hbar^{p}\eps,\\
\tilde{F}^jTQ\widehat{\Pol}(M,-1)&:= \tilde{F}^jQ\widehat{\Pol}(M,-1)\oplus \prod_{p \ge j}F_p\sD_{A/R}(M)\hbar^{p}\eps,
\end{align*}
for $\eps$ of degree $0$ with $\eps^2=0$. The Lie bracket is given by $ [u+v\eps, x+y\eps]= [u,x]+ [u,y]\eps + [v,x]\eps$.
\end{definition}

\begin{definition}\label{TQpoldef}
Given a Maurer--Cartan element $\Delta \in  \mc(\tilde{F}^2Q\widehat{\Pol}(M,-1)) $, define the centre of $(M,\Delta)$ by
\[
 T_{\Delta}Q\widehat{\Pol}(M,-1):= \prod_{p \ge 0}F_p\sD_{A/R}(M)\hbar^{p},
\]
with derivation $\delta + [\Delta,-]$ (necessarily square-zero by the Maurer--Cartan conditions). 
 
Multiplication of differential operators  makes this a  DGAA (with no need to rescale the product by $\hbar$), and it has a filtration
\[
 \tilde{F}^iT_{\Delta}Q\widehat{\Pol}(M,-1):= \prod_{p \ge i}F_p\sD_{A/R}(M)\hbar^{p},
\]
with $\tilde{F}^i \cdot \tilde{F}^j\subset \tilde{F}^{i+j}$.  Given $\Delta \in  \mc(F^2Q\widehat{\Pol}(M,-1)/\tilde{F}^p)$, we define $T_{\Delta}Q\widehat{\Pol}(M,-1)/\tilde{F}^p$ similarly ---  this is also a DGAA as $\tilde{F}^p$ is an ideal.
\end{definition}

Observe that $T_{\Delta}Q\cP(M,-1):= \mmc(\tilde{F}^2T_{\Delta}Q\widehat{\Pol}(M,-1))$ is just the fibre of $\mmc( \tilde{F}^2TQ\widehat{\Pol}(M,-1)) \to  \mmc(\tilde{F}^2 Q\widehat{\Pol}(M,-1))$ over $\Delta$.

Similarly to Definition \ref{Gdef}, there is a filtration $G$ on  
$TQ\widehat{\Pol}(M,-1), T_{\Delta}Q\widehat{\Pol}(M,-1) $ given by powers of $\hbar$. Since $\gr_G^i\tilde{F}^{p-i}Q\widehat{\Pol}= \prod_{j \ge p-i} \gr^F_{j-i}\sD_{A/R}(M)\hbar^{j-1}$,
the associated gradeds of the filtration admit maps 
  \begin{align*}
 \gr_G^i\tilde{F}^pQ\widehat{\Pol}(M,-1) &\to \prod_{j \ge p} \HHom_A(\CoS_A^{j-i}(\Omega^1_{A/R})\ten_AM,M)\hbar^{j-1}\\
\gr_G^i\tilde{F}^pT_{\Delta}Q\widehat{\Pol}(M,-1) &\to \prod_{j \ge p} \HHom_A(\CoS_A^{j-i}(\Omega^1_{A/R})\ten_AM,M)\hbar^{j}.
\end{align*}
which are isomorphisms when $A$ is semi-smooth (in particular whenever $A$ is cofibrant as a CDGA over $R$).

For the filtration $F$ of Definition \ref{Fdef}, we may rewrite these maps as
 \begin{align*}
 \gr_G^i\tilde{F}^pQ\widehat{\Pol}(A,-1) &\to  F^{p-i}\widehat{\Pol}(A,-1)\hbar^{i},\\ 
\gr_G^i\tilde{F}^pT_{\Delta}Q\widehat{\Pol}(M,-1) &\to F^{p-i}T_{\pi_{\Delta}}\widehat{\Pol}(A,-1)\hbar^{i},
\end{align*}
where $\pi_{\Delta} \in \mc(F^2\widehat{\Pol}(A,-1))$ denotes the image of $\Delta$ under the map  $ \gr_G^0\tilde{F}^2Q\widehat{\Pol}(A,-1) \to  F^2\widehat{\Pol}(A,-1)$. 

Since the cohomology groups of $T_{\pi_{\Delta}}\widehat{\Pol}(A,-1)$ are Poisson  cohomology, we will refer to the cohomology groups of  $T_{\Delta}Q\widehat{\Pol}(M,-1)$ as quantised Poisson cohomology.

We write  $T_{\Delta}Q^{tw}\cP(M,-1):= \mmc((G*\tilde{F})^2T_{\Delta}Q\widehat{\Pol}(M,-1))$.

\begin{definition}\label{Qnonnegdef}
 Say that an $E_0$ quantisation $\Delta =  \sum_{j \ge 2} \Delta_j \hbar^{j}$ is non-degenerate if the map 
\[
\Delta_2^{\sharp}\co  M\ten_A \Omega^1_A \to \HHom_A(\Omega^1_A,M)[1]
\]
is a quasi-isomorphism and $\Omega^1_A $ is perfect.
\end{definition}

\begin{definition}\label{TQPdef}
Define the tangent spaces
\begin{eqnarray*}
 TQ\cP(M,-1)&:=& \Lim_i \mmc( \tilde{F}^2 TQ\widehat{\Pol}(M,-1)/\tilde{F}^{i+2}),\\
TQ^{tw}\cP(M,-1)&:=& \Lim_i \mmc( (G*\tilde{F})^2 TQ\widehat{\Pol}(M,-1)/\tilde{F}^{i+2}),\\
\end{eqnarray*}
with $TQ\cP(M,-1)/G^k$, $TQ^{tw}\cP(M,-1)/G^k $ defined similarly.
 \end{definition}
These are simplicial sets over $Q\cP(M,-1)$ (resp.  $Q^{tw}\cP(M,-1)$, $Q\cP(M,-1)/G^k$, $Q^{tw}\cP(M,-1)/G^k$), fibred in simplicial abelian groups. 

\begin{definition}\label{Qsigmadef}
Define the canonical tangent vector
\[
  \sigma=-\pd_{\hbar^{-1}}\co Q\widehat{\Pol}(M,-1) \to TQ\widehat{\Pol}(M,-1)
\]
 by $\alpha \mapsto \alpha + \eps \hbar^{2}\frac{\pd \alpha}{\pd \hbar}$. Note that this is a morphism of filtered  DGLAs, so gives a map $ \sigma \co Q\cP(M,-1)\to TQ\cP(M,-1)$, with $\sigma(\Delta) \in \z^1(\tilde{F}^2T_{\Delta}Q\widehat{\Pol}(M,-1))$. 
\end{definition}

\subsection{Generalised pre-symplectic structures}\label{DRsn}

We now develop  generalised shifted pre-symplectic structures as certain power series in the de Rham complex, leading to a notion of compatibility between these generalised structures and quantisations. 

\begin{definition}
Define the de Rham complex $\DR(A/R)$ to be the product total complex of the bicomplex
\[
 A \xra{d} \Omega^1_{A/R} \xra{d} \Omega^2_{A/R}\xra{d} \ldots,
\]
so the total differential is $d \pm \delta$.

We define the Hodge filtration $F$ on  $\DR(A/R)$ by setting $F^p\DR(A/R) \subset \DR(A/R)$ to consist of terms $\Omega^i_{A/R}$ with $i \ge p$. In particular, $F^p\DR(A/R)= \DR(A/R)$ for $p \le 0$.
\end{definition}

\begin{definition}
When $A$ is a cofibrant CDGA over $R$, recall that  a $(-1)$-shifted pre-symplectic structure $\omega$ on $A/R$ is an element
\[
 \omega \in \z^{1}F^2\DR(A/R).
\]
In \cite{PTVV}, shifted pre-symplectic structures are referred to as closed $2$-forms.

A $(-1)$-shifted pre-symplectic structure $\omega$ is called symplectic if $\omega_2 \in \z^{-1}\Omega^2_{A/R}$ induces a quasi-isomorphism
\[
 \omega_2^{\sharp} \co \Hom_A(\Omega^1_{A/R},A) \to \Omega^1_{A/R}[-1],
\]
and $\Omega^1_{A/R} $ is perfect as an $A$-module.
\end{definition}

In order to define compatibility functors for quantisations, we will need to construct $A_{\infty}$-morphisms from the de Rham algebra, which we will do using the following DGAA resolution.

\begin{definition}\label{DRprimedef}
Write $A^{\ten \bt +1}$ for the cosimplicial CDGA $n \mapsto A^{n +1}$ given by the \v Cech nerve, with $I$ the kernel of the diagonal map $A^{\ten \bt +1}\to A$. This has a filtration $F$ given by powers $F^p :=(I)^p$ of $I$, and we define the filtered cosimplicial CDGA $\hat{A}^{\ten \bt +1}$ to be the completion
\begin{align*}
 \hat{A}^{\ten \bt +1}&:= \Lim_qA^{\ten \bt +1}/F^q,\\
F^p \hat{A}^{\ten \bt +1}&:= \Lim_qF^p/F^q.
\end{align*}

We then take the Dold--Kan conormalisation $N\hat{A}^{\bt +1}$, which becomes a filtered bi-DGAA via the  Alexander--Whitney cup product. Explicitly,  $N^n \hat{A}^{\bt +1}$ is the intersection of the kernels of all the big diagonals $  \hat{A}^{n +1}\to \hat{A}^{n}$, and the cup  product is given by 
\[
 (a_0 \ten \ldots \ten a_m) \smile (b_0 \ten \ldots \ten b_n)= a_0 \ten \ldots \ten a_{m-1}\ten (a_mb_0) \ten b_1 \ten \ldots \ten b_n. 
\]

 We then define $\DR'(A/R)$ to be the product total complex
\[
 \DR'(A/R):=\Tot^{\Pi} N\hat{A}^{\bt +1}
\]
regarded as a filtered DGAA over $R$, with $F^p\DR'(A/R):=\Tot^{\Pi} NF^p\hat{A}^{\bt +1}$.
\end{definition}

The following is standard:
\begin{lemma}\label{cfDRlemma}
 There is a filtered quasi-isomorphism $\DR'(A/R) \to \DR(A/R)$, given by $N^n\hat{A}^{\bt +1}\to 
N^n\hat{A}^{\bt +1}/F^{n+1} \cong (\Omega^1_{A/R})^{\ten_A n} \to  \Omega^n_{A/R}$. 
\end{lemma}
\begin{proof}
 It suffices to show that the map is a quasi-isomorphism on the graded pieces associated to the filtration. Now, $\gr_F^p\hat{A}^{\ten \bt +1}= \Symm_A^p D(\Omega^1_A[-1])$, where $D$ denotes Dold--Kan denormalisation from cochain complexes to cosimplicial complexes. Thus
\[
 N\gr_F^p\hat{A}^{\ten \bt +1}= N\Symm_A^p D(\Omega^1_A[-1]),
\]
so $\Tot^{\Pi}N\Symm_A^p D(\Omega^1_A[-1])$ is quasi-isomorphic to $\Symm_A^p ND(\Omega^1_A[-1])= \Symm_A^p(\Omega_A^1[-1])= \Omega^p_A[-p]$, combining Dold--Kan with Eilenberg--Zilber.
\end{proof}

\begin{definition}\label{tildeFDRdef}
Define a decreasing filtration $\tilde{F}$ on $ \DR'(A/R)\llbracket\hbar\rrbracket$ by 
\[
 \tilde{F}^p\DR'(A/R):= \prod_{i\ge 0} F^{p-i}\DR'(A/R)\hbar^{i},
\]
where we adopt the convention that $F^j\DR'=\DR'$ for all $j\le 0$.

Define  further filtrations $G, G*\tilde{F}$ by $ G^k \DR'(A/R)\llbracket\hbar\rrbracket = \hbar^{k}\DR'(A/R)\llbracket\hbar\rrbracket$, and $(G*\tilde{F})^p:= \sum_{i+j=p}G^i\cap\tilde{F}^j$, so
\[
 (G*\tilde{F})^p= \prod_{i\ge 0} F^{p-2i}\DR'(A/R)\hbar^{i}.
\]
\end{definition}

This makes  $(\DR'(A/R)\llbracket\hbar\rrbracket,G*\tilde{F}) $ into a filtered DGAA, since $\tilde{F}^p\tilde{F}^q \subset \tilde{F}^{p+q}$ and similarly for $G$.

\begin{definition}
Define a generalised $(-1)$-shifted pre-symplectic structure on a cofibrant CDGA $A/R$ to be an element
\[
 \omega \in \z^1((G*\tilde{F})^2\DR'(A/R)\llbracket\hbar\rrbracket) = \z^1(F^2\DR'(A/R)) \by \hbar \z^1\DR'(A/R)\llbracket\hbar\rrbracket.
\]
Call this symplectic if $\Omega^1_{A/R}$ is perfect as an $A$-module and the  leading term $\omega_0 \in \z^1F^2\DR'(A/R)$ induces a quasi-isomorphism
\[
 [\omega_0]^{\sharp} \co \Hom_A(\Omega^1_{A/R},A) \to \Omega^1_{A/R}[-1],
\]
for $[\omega_0] \in \z^{-1}\Omega^2_{A/R}$ the image of $\omega_0$ modulo $F^3$.
\end{definition}

\begin{definition}\label{GPreSpdef}
 Define the space of generalised $(-1)$-shifted pre-symplectic structures on $A/R$ to be the simplicial set
\[
 G\PreSp(A/R,-1):= \Lim_i\mmc( (G*\tilde{F})^2\DR'(A/R)\llbracket\hbar\rrbracket/\tilde{F}^{i+2}), 
\]
where we regard the cochain complex  $\DR'(A/R)$ as a  DGLA with trivial bracket. Write $\PreSp = G\PreSp/G^1$.

Also write $G\PreSp(A/R,-1)/\hbar^{k}:= \Lim_i\mmc( ((G*\tilde{F})^2\DR'(A/R)[\hbar]/(G^k +\tilde{F}^{i+2)} )$, so $ G\PreSp(A/R,-1)= \Lim_k G\PreSp(A/R,-1)/\hbar^{k} $.

Set $G\Sp(A/R,-1) \subset G\PreSp(A/R,-1)$ to consist of the symplectic structures --- this is a union of path-components.
\end{definition}
Note that  $G\PreSp(A/R,-1)$  is canonically weakly  equivalent to the Dold--Kan denormalisation of the good truncation complex $\tau^{\le 0}((G*\tilde{F})^2\DR(A/R)\llbracket\hbar\rrbracket[1])$ (and similarly for the various quotients we consider),   but the description in terms of $\mmc$ will simplify comparisons. In particular, we have
\[
 \pi_iG\PreSp(A/R,-1)\cong \H^{1-i}(F^2\DR(A/R)) \by \hbar \H^{1-i}(\DR(A/R))\llbracket\hbar\rrbracket. 
\]

\subsubsection{Compatible quantisations}

We will now develop the notion of compatibility between a (truncated) generalised pre-symplectic structure and a (truncated) $E_0$  quantisation. For the $0$th order truncation (corresponding to $k=1$ in Definition \ref{Qcompdef}), this
 recovers the notion of compatibility between pre-symplectic and Poisson structures from \cite{poisson}.

\begin{lemma}\label{mulemma1}
 Take  a complete    filtered  graded-associative $R$-algebra $(B,\Fil^{\bt})$ and a  morphism $\phi \co A^{\#} \to \Fil^0B$ of graded $R$-algebras; assume that the left and right $A$-module structures on $\gr_{\Fil}B$ agree. Then for any  $\Delta \in \Fil^0B^1$,  there is an  
associated morphism 
\[
 \mu(-,\Delta) \co (\DR'(A)^{\#},F^{\bt}) \to (B,\Fil^{\bt})
\]
of filtered  graded-associative $R$-algebras
induced by the graded algebra map on $A^{\ten \bt +1}$ determined by  $\mu(1\ten 1, \Delta) = \Delta$ and $\mu(a, \Delta)=a$ for $ a \in A$.

Given $\rho \in \Fil^jB^k$, there is then a filtered $R$-linear derivation
\[
 \nu(-, \Delta, \rho) \co (\DR'(A/R)^{\#},F^{\bt}) \to (B[k], \Fil^{\bt+j})
\]
induced by the $\mu(-,\Delta)$-derivation on $A^{\ten \bt +1}$ determined by $\nu(1\ten 1, \Delta, \rho)= \rho$. 
\end{lemma}
\begin{proof}
 Explicitly, $\mu(-,\Delta)$ is given on $A^{\ten \bt +1}$ by
 \[
 (a_0\ten a_1 \ten \ldots \ten a_r) \mapsto a_0\Delta a_1\Delta \ldots \Delta a_r,
\]
because $a\ten b= a\smile (1\ten 1) \smile b$ and so on. Similarly, $\nu(-, \Delta, \rho)$  is given by
\[
 (a_0\ten a_1 \ten \ldots \ten a_r) \mapsto \sum_{i=0}^{r-1} \pm a_0\Delta a_1\Delta \ldots  \Delta a_i \rho a_{i+1}\Delta \ldots \Delta a_r.
\]
We need to show that these respect the filtrations, so giving maps on  $\hat{A}^{\ten \bt +1}$ and hence filtered morphisms on $\DR(A/R)'$. 

Observe that the filtration on $  A^{\ten \bt +1}$ is generated by that on $A^{\ten 2}$, in the sense that
\[
 F^pA^{\ten m+1} = \sum_{p_1+ \ldots + p_m=p} (F^{p_1}A^{\ten 2}) \smile \ldots \smile (F^{p_m}A^{\ten 2}).
\]
It therefore suffices to show that $\mu(-,\Delta) \co A^{\ten 2} \to B$  and $\nu(-,\Delta,\rho)$ are appropriately filtered. 

Writing $[x,y]:= x \smile y - (-1)^{\deg x\deg y}y \smile x$ and $\cdot$ for the internal multiplication  on $A^{\ten r}$, it follows that   for $a \in A$ and $x \in A \ten A$ we have  $[a,x]= (a \ten 1 \mp 1 \ten a)\cdot x$. Since $I= \ker(A^{\ten 2} \to A)$ is generated by elements of the form $ (a \ten 1 \mp 1 \ten a)$, this means that $[A,J]= I \cdot J$ for all ideals $J$ in $A^{\ten 2}$. Because $F^p= I^{\cdot p}$, this gives
\[
 F^pA^{\ten 2} = \underbrace{[A,[A,\ldots [A,}_p A^{\ten 2}]\ldots ],
\]
so $F$ is the smallest multiplicative filtration for which left and right $A$-modules structures on $\gr_F  A^{\ten \bt +1}$ agree. 

Therefore the algebra homomorphism $\mu(-,\Delta)$ must send $F^p$ to $\Fil^p$, and the derivation $\nu(-, \Delta, \rho)$ must send $F^p$ to $\Fil^{p+j}$; in particular, the maps descend to the completion $\DR'(A/R)$.
\end{proof}

\begin{lemma}\label{QPolmudef}
 Given  $\Delta \in ((G*\tilde{F})^2Q\widehat{\Pol}(M,-1)/G^k)^{1}$, Lemma \ref{mulemma1}  gives a   morphism 
\[
 \mu(-,\Delta) \co \DR'(A/R)[\hbar]/G^k \to T_{\Delta}Q\widehat{\Pol}(M,-1)/G^k
\]
of graded associative $R[\hbar]/\hbar^{k}$-algebras, respecting the filtrations $(G*\tilde{F})$.

Given $\rho \in ((G*\tilde{F})^pT_{\Delta}Q\widehat{\Pol}(M,-1)/G^k)^r$, there is then a $R[\hbar]/\hbar^{k}$-linear derivation
\[
 \nu(-, \Delta, \rho) \co (\DR'(A/R)[\hbar]/\hbar^{k},(G*\tilde{F})^{\bt})  \to (T_{\Delta}Q\widehat{\Pol}(M,-1)[r]/G^k,(G*\tilde{F})^{\bt+p}).
\]
\end{lemma}
\begin{proof}
It suffices to prove this for the limit over all $k$, as $\Delta$ and $\rho$ always lift to $(G*\tilde{F})^2Q\widehat{\Pol}(M,-1)$.
Set $T= T_0Q\widehat{\Pol}(M,-1)[\hbar^{-1}]$, with filtrations $\tilde{F}$ given by powers of $\hbar$  and $G^iT:= \hbar^i T_0Q\widehat{\Pol}(M,-1)$, so $G^i\tilde{F}^jT= \prod_{p \ge j} \hbar^pF_{p-i}$.
These filtrations are multiplicative, with $[G^i,G^j] \subset G^{i+j-1}$, so the convolution filtration also satisfies $[(G*\tilde{F})^i,(G*\tilde{F})^j] \subset (G*\tilde{F})^{i+j-1} $; in particular $\gr_{G*\tilde{F}}T$ is commutative, so its left and right $A$-module structures agree (the same is not true of $\gr_{\tilde{F}}T$, which makes the convolution filtration necessary). Explicitly, $(G*\tilde{F})^pT= \prod_k \hbar^k F_{2k-p}$.

Then $Q\widehat{\Pol}(M,-1)= G^{-1}T $ and $T_0Q\widehat{\Pol}(M,-1)= G^0T$, with      $ G^i\tilde{F}^jT=G^i\tilde{F}^jT_0Q\widehat{\Pol}(M,-1) $ and $G^{i-1}\tilde{F}^{j-1}T=G^i\tilde{F}^jQ\widehat{\Pol}(M,-1)  $ whenever $i \ge 0$. Thus $(G*\tilde{F})^p Q\widehat{\Pol}\subset  (G*\tilde{F})^{p-2}T$, so in particular $\Delta$ lies in $(G*\tilde{F})^0T$.
 Lemma \ref{mulemma1} now gives  filtered morphisms
\begin{align*}
 \mu(-,\Delta) \co (\DR'(A/R),F^{\bt}) &\to (T, (G*\tilde{F})^{\bt})\\
 \nu(-, \Delta, \rho) \co(\DR'(A/R),F^{\bt}) &\to (T[r], (G*\tilde{F})^{\bt+p}).
\end{align*}

Since $\Delta \in G^{-1}T$, we have $[\Delta, a] \in G^0T$ for all $a \in A$, so $\mu(F^1(A\ten A), \Delta)\subset G^0T$. Since $F^1(A\ten A)$ topologically generates  $\DR'(A/R)$ under multiplication, $\mu(\DR'(A/R), \Delta)$ is thus contained in  the subalgebra $G^0T= T_0Q\widehat{\Pol}(M,-1)$ of $T$. 
Extending linearly gives a map from $\DR'(A/R)\llbracket \hbar \rrbracket$;
since $\hbar G^i\tilde{F}^qT\subset G^{i+1}\tilde{F}^{q+1}T$, we then see that  $\mu(\hbar^i F^{p-2i} , \Delta) \subset G^{p-r-i}\tilde{F}^{r+i}$, so
\[
 \mu((G*\tilde{F})^p\DR(A/R)'\llbracket\hbar\rrbracket, \Delta) \subset (G*\tilde{F})^pT_{\Delta}Q\widehat{\Pol}(M,-1),
\]
and similarly for $\nu$.
\end{proof}

\begin{lemma}\label{keylemma}
Take a complete    filtered   $R$-DGAA $(B,\Fil^{\bt})$ and a  morphism $\phi \co A \to \Fil^0B$ of  $R$-DGAAs; assume that the left and right $A$-module structures on $\gr_{\Fil}B$ agree. Then for any  $\Delta \in\Fil^0B^1$,  and  any $\omega \in \DR'(A/R)$, we have
\begin{align*}
[\Delta,\mu(\omega, \Delta)] = \mu(d\omega, \Delta) + \half \nu(\omega, \Delta, [\Delta,\Delta]),\\
\delta_{\Delta}\mu(\omega, \Delta) = \mu(D\omega, \Delta) + \nu(\omega, \Delta,\kappa(\Delta)),
\end{align*}
where $\delta$ are the structural differentials on $A,B$, with   $\delta_{\Delta}:= \delta + [\Delta,-]$,  $D:= d \pm \delta$  the total differential on $\DR'(A/R)$, and $\kappa(\Delta):= [\delta, \Delta] + \half[\Delta,\Delta]$.
\end{lemma}
\begin{proof}
Both $[\Delta,\mu(-, \Delta)]$ and $\delta_{\Delta}\mu(-, \Delta)$ are derivations with respect to $\mu(-, \Delta)$, so it suffices to verify these identities on the generators $a, df$   of the denormalisation of $\DR'(A/R)$, for $a,f \in A$. 

In these cases, we have
\begin{align*}
[\Delta, \mu(a,\Delta)] = [\Delta, a] = \mu(1\ten a \mp a\ten 1, \Delta)&= \mu(da, \Delta),\\
 [\Delta,\mu(1\ten f -f\ten 1, \Delta)]= \Delta^2f\mp f\Delta^2&= \half \nu(df, \Delta, [\Delta,\Delta]).
\end{align*}
Because $\nu(a, \Delta, [\Delta,\Delta])=0 $ ($\nu$ being $A$-linear) 
 and $ddf=0$, this gives the required results, the second set of equalities following by adding $\delta$.
\end{proof}

In particular, Lemma \ref{keylemma} implies that when $\Delta\in \mc((G*\tilde{F})^2Q\widehat{\Pol}(A,n)/G^k)$ is an $E_0$ quantisation, $\mu(-,\Delta)$ is a chain map (since $T_{\Delta}Q\widehat{\Pol}(M,-1) = (T_0Q\widehat{\Pol}(M,-1), \delta_{\Delta})$), so $\mu(-,\Delta)$ defines a map from de Rham cohomology to quantised Poisson cohomology.

\begin{definition}
We say that a generalised  $(-1)$-shifted  pre-symplectic structure $\omega$ and an $E_0$ quantisation $\Delta$ of a strict line bundle $M$  are  compatible (or a compatible pair) if 
\[
 [\mu(\omega, \Delta)] = [-\pd_{\hbar^{-1}}(\Delta)] \in  \H^1((G*\tilde{F})^2T_{\Delta}Q\widehat{\Pol}(M,-1)),
\]
where $\sigma=-\pd_{\hbar^{-1}}$ is the canonical tangent vector of Definition \ref{Qsigmadef}. 
\end{definition}

\begin{definition}\label{vanishingdef}
Given a simplicial set $Z$, an abelian group object $A$ in simplicial sets over $Z$,  a space $X$ over $Z$ and a morphism  $s \co X \to A$ over $Z$, define the homotopy vanishing locus of $s$ over $Z$ to be the homotopy limit of the diagram
\[
\xymatrix@1{ X \ar@<0.5ex>[r]^-{s}  \ar@<-0.5ex>[r]_-{0} & A \ar[r] & Z}.
\]
\end{definition}

\begin{definition}\label{Qcompdef}
Define the space $Q\Comp(M,-1)$ of compatible quantised $(-1)$-shifted pairs to be the homotopy vanishing locus of  
\[
 (\mu - \sigma) \co G\PreSp(A/R,-1) \by Q\cP(M,-1) \to TQ^{tw}\cP(M,-1)
\]
over $Q^{tw}\cP(M,-1)$

We define a cofiltration on this space by setting $ Q\Comp(M,-1)/G^k$ to be the homotopy vanishing locus of  
\[
 (\mu - \sigma) \co (G\PreSp(A/R,-1)/G^k)  \by (Q\cP(M,-1)/G^k)  \to TQ^{tw}\cP(M,-1)/G^k 
\]
over $Q^{tw}\cP(M,-1)/G^k $.

\end{definition}

When $k=1$, note that this recovers the notion of compatible $(-1)$-shifted pairs from \cite{poisson}.

\begin{definition}
 Define $Q\Comp(M,-1)^{\nondeg} \subset Q\Comp(M,-1)$ to consist of compatible quantised pairs $(\omega, \Delta)$ with $\Delta$ non-degenerate. This is a union of path-components, and by \cite[Lemma \ref{poisson-compatnondeg}]{poisson} has a natural map 
\[
 Q\Comp(M,-1)^{\nondeg}\to G\Sp(A/R,-1)
\]
as well as the canonical map
\[
 Q\Comp(M,-1)^{\nondeg} \to Q\cP(M,-1)^{\nondeg}.
\]
\end{definition}

\begin{proposition}\label{QcompatP1} 
For any strict line bundle $M$, the canonical map
\begin{eqnarray*}
    Q\Comp(M,-1)^{\nondeg} \to  Q\cP(M,-1)^{\nondeg}           
\end{eqnarray*}
 is a weak equivalence. In particular, there is a morphism
\[
  Q\cP(M,-1)^{\nondeg} \to G\Sp(A/R,-1)
\]
in the homotopy category of simplicial sets.
\end{proposition}
\begin{proof}
We adapt the proof of \cite[Proposition \ref{poisson-compatP1}]{poisson}.
For any $\Delta \in Q\cP(M,-1)$, the homotopy fibre of $Q\Comp(A/R,-1)^{\nondeg} $ over $\Delta$ is just the homotopy fibre of
\[
\mu(-,\Delta)  \co G\PreSp(A/R,-1)  \to T_{\Delta}Q^{tw}\cP(M,-1) 
\]
over $-\pd_{\hbar^{-1}}(\Delta)$.

The map $\mu(-,\Delta) \co \DR'(A/R)\llbracket\hbar\rrbracket \to T_{\Delta}Q\widehat{\Pol}(M,-1)$ is a morphism of complete $G*\tilde{F}$-filtered $R\llbracket\hbar\rrbracket$-DGAAs by Lemma \ref{keylemma}. Since the morphism is $R\llbracket\hbar\rrbracket$-linear, it maps $G^k(G*\tilde{F})^p\DR'(A/R)\llbracket\hbar\rrbracket$ to $   G^k(G*\tilde{F})^pT_{\Delta}Q\widehat{\Pol}(M,-1)$. Non-degeneracy of $\Delta_2$ modulo $F_1$ implies that $\mu(-,\Delta)$ induces  quasi-isomorphisms
\[
  \Omega^{p-2k}\hbar^{k}[2k-p] \to \HHom_A(\CoS^{p-2k}\Omega^1_{A/R}\ten_AM, M)\hbar^{p-k}
\]
on the associated gradeds $\gr_G^k\gr_{(G*\tilde{F})}^p$.  We therefore have a quasi-isomorphism of bifiltered complexes, so we have isomorphisms on homotopy groups:
\begin{eqnarray*}
 \pi_jG\PreSp(A/R,-1)  &\to& \pi_jT_{\Delta}Q^{tw}\cP(M,-1)\\
 \H^{1-j}((G*\tilde{F})^2 \DR(A/R)\llbracket\hbar\rrbracket) &\to&  \H^{1-j}((G*\tilde{F})^2T_{\Delta}Q\widehat{\Pol}(M,-1)).
\end{eqnarray*}
\end{proof}

\subsection{Comparing quantisations and generalised symplectic structures}\label{comparisonsn}

We now investigate the extent to which the elements of a compatible pair determine each other.

\begin{definition}\label{Ndef}
Given a compatible pair  $(\omega, \pi) \in  \Comp(A,-1)= Q\Comp(M,-1)/G^1$, and $k \ge 0$, define the complex 
$
 N(\omega,\pi,k) 
$
to be the cocone of the map 
\begin{align*}
 \gr_G^k(G*\tilde{F})^2(\DR'(A/R)\llbracket\hbar\rrbracket\oplus \gr_G^k\tilde{F}^2Q\widehat{\Pol}(M,-1) 
 \to \gr_G^k(G*\tilde{F})^2T_{\pi}Q\widehat{\Pol}(M,-1)
\end{align*}
given by combining                                
\begin{align*}
 \gr_G^k  \mu(-,\pi) \co  \gr_G^k(G*\tilde{F})^2\DR'(A/R)\llbracket\hbar\rrbracket &\to  \gr_G^k(G*\tilde{F})^2T_{\pi}Q\widehat{\Pol}(M,-1) \\ 
F^{2-2k}\DR(A/R)\hbar^{k} &\to\prod_{i \ge (2-2k),0} \HHom_A(\CoS_A^{i}(\Omega^1_{A/R}),A)\hbar^{i+k}
\end{align*}
with the maps 
\begin{align*}
 \gr_G^k\nu(\omega, \pi) + \pd_{\hbar^{-1}} \co (\gr_G^k\tilde{F}^2Q\widehat{\Pol}(M,-1), \delta_{\pi}) &\to \gr_G^k(G*\tilde{F})^2T_{\pi}Q\widehat{\Pol}(M,-1)\\
\prod_{i \ge (2-k),0}\HHom_A(\CoS_A^{i}(\Omega^1_{A/R}),A)\hbar^{i+k-1}  &\to  \prod_{i \ge (2-2k),0} \HHom_A(\CoS_A^{i}(\Omega^1_{A/R}),A)\hbar^{i+k},
\end{align*}
 where                              
\[
 \nu(\omega, \pi)(b):= \nu(\omega, \pi, b).
\]
\end{definition}

It follows from the proof of  Proposition \ref{QcompatP1} that  the maps $\gr_G^k  \mu(-,\pi)$ are all $F$-filtered quasi-isomorphisms when $\pi$ is non-degenerate, so the projection maps $N(\omega,\pi,k) \to \gr_G^k\tilde{F}^2Q\widehat{\Pol}(M,-1)$ are also quasi-isomorphisms. The behaviour of the other projection is more subtle for low $k$, but it behaves well thereafter:

\begin{lemma}\label{tangentlemma}
The projection maps 
\[
 N(\omega,\pi,k) \to \hbar^{k}\DR(A/R)
\]
 are $F$-filtered quasi-isomorphisms for all $k \ge 2$.
\end{lemma}
\begin{proof}
This amounts to showing that the map 
\[
 \gr_G^k\nu(\omega, \pi) + \pd_{\hbar^{-1}}
\]
is a filtered quasi-isomorphism. 
 It suffices to show that the associated maps
 \begin{align*}
 \gr_F^p\gr_G^k\nu(\omega, \pi) + \pd_{\hbar^{-1}} \co \gr_F^p\gr_G^k(G*\tilde{F})^2Q\widehat{\Pol}(M,-1) &\to \gr_F^p\gr_G^k(G*\tilde{F})^2T_{\pi}Q\widehat{\Pol}(M,-1)\\
\HHom_A(\CoS_A^{p}(\Omega^1_{A/R}),A)\hbar^{p+k-1}  &\to  \HHom_A(\CoS_A^{p}(\Omega^1_{A/R}),A)\hbar^{p+k},
\end{align*}
are quasi-isomorphisms for all $p \ge 0$.

Reasoning as in \cite[Lemma \ref{poisson-nondegtangent}]{poisson}, $\gr_{F}\gr_G\nu(\omega, \pi)$ is an $R\llbracket\hbar\rrbracket$-linear  derivation on $\gr_{F}\gr_GQ\widehat{\Pol}(M,-1) \cong \widehat{\Pol}(A,-1)\llbracket\hbar\rrbracket$ with respect to the commutative multiplication. It is given on generators $\HHom_A(\Omega^1_{A/R},A)$ by $\hbar\pi^{\sharp} \circ \omega^{\sharp}$. Compatibility of  $\omega$ and $\pi$ implies that $\pi^{\sharp} \circ \omega^{\sharp}$ is homotopy idempotent  by \cite[Example \ref{poisson-compatex1}]{poisson}. Thus $ \hbar^{-1}\gr_{F}^p\gr_G^k\nu(\omega, \pi)$ is homotopy diagonalisable, with integral eigenvalues in the interval $[0, p]$.

On the other hand, $\pd_{\hbar^{-1}}$ coincides on $\gr_{F}^p\gr_G^k $ with multiplication by $(1-p-k)\hbar $, so the eigenvalues of  $ \hbar^{-1}\gr_{F}^p\gr_G^k\nu(\omega, \pi) + \hbar^{-1}\pd_{\hbar^{-1}}$ lie in $[1-p-k, 1-k] $, giving a  quasi-isomorphism when $k > 1$.
\end{proof}


\begin{proposition}\label{quantprop}
The   maps
\begin{align*}
 Q\cP(M,-1)^{\nondeg}/G^k &\to (Q\cP(M,-1)^{\nondeg}/G^2)\by^h_{(G\Sp(A,-1)/G^2)}(G\Sp(A,-1)/G^k) \\ 
&\simeq (Q\cP(M,-1)^{\nondeg}/G^2)\by \prod_{i=2}^{k-1} \mmc(\DR(A/R)\hbar^i)
\end{align*}
coming from Proposition \ref{QcompatP1}  are weak equivalences for all $k \ge 2$.
\end{proposition}
\begin{proof}
Proposition \ref{QcompatP1} gives equivalences between $Q\cP^{\nondeg}$ and $Q\Comp^{\nondeg}$.  Fix $(\omega, \pi) \in \Comp(A,-1)$ and denote homotopy fibres by subscripts.
Arguing as in  the proof of \cite[Proposition \ref{poisson-compatobs}]{poisson}, but 
with abelian (rather than central) extensions of DGLAs as in \cite[Lemma \ref{dmc-obsdgla}]{dmc}  gives a commutative diagram  
\[
\begin{CD}
 (Q\Comp(M,-1)/G^{k+1})_{(\omega, \pi)} @>>> (Q\Comp(M,-1)/G^k)_{(\omega,\pi)} @>>> \mmc(N(\omega,\pi,k)[1])\\
@VVV @VVV @VVV \\
(G\PreSp(A,-1)/G^{k+1})_{\omega}@>>> (G\PreSp(A,-1)/G^{k})_{\omega} @>>> \mmc(F^{2-2k}\hbar^{k}\DR(A/R)[1])
\end{CD}
\]
of  fibre sequences.

The right-hand map is a weak equivalence for $k \ge 2$, by Lemma \ref{tangentlemma}, so $ Q\Comp(M,-1)/G^{k+1} $ is equivalent to the homotopy fibre product 
\[
 (Q\Comp(M,-1)/G^k)\by^h_{G\PreSp(A,-1)/G^{k} }G\PreSp(A,-1)/G^{k+1},
\]
and the result follows by induction.
\end{proof}

\begin{remark}\label{quantrmk}
 Taking the limit over all $k$, Proposition \ref{quantprop}  gives an equivalence
\[
  Q\cP(M,-1)^{\nondeg} \simeq (Q\cP(M,-1)^{\nondeg}/G^2)\by \prod_{i \ge 2} \mmc(\DR(A/R)\hbar^i);
\]
in particular, this means that there is a canonical map 
\[
 (Q\cP(M,-1)^{\nondeg}/G^2) \to Q\cP(M,-1)^{\nondeg},
\]
corresponding to the distinguished point $0 \in \mmc( \hbar^2\DR(A/R)\llbracket\hbar\rrbracket)$.

Thus to quantise a non-degenerate $(-1)$-shifted Poisson structure $\pi =\sum_{j \ge 2} \pi_j$ (or equivalently, by \cite[Corollary \ref{poisson-compatcor2}]{poisson}, a $(-1)$-shifted symplectic  structure), it suffices to lift the power series $\sum_{j \ge 2} \pi_j (-\hbar)^{j-1}$ to a Maurer--Cartan element of $\prod_{j \ge 2} (F_j\sD(M)/F_{j-2})\hbar^{j-1}$.

Even in the degenerate case, the proof of Proposition \ref{quantprop} gives a sufficient first-order criterion  for quantisations to exist:
\[
 Q\Comp(M,-1) \simeq (Q\Comp(M,-1)/G^2)\by \prod_{i \ge 2} \mmc(\DR(A/R)\hbar^i).
\]
\end{remark}

\section{Quantisation for derived DM $N$-stacks}\label{DMsn}

In this section, we will globalise the results of the previous section to the setting of derived DM $N$-stacks.
In order to pass from derived affine schemes to derived DM stacks, we will exploit \'etale functoriality using Segal spaces.

The basic idea is that given a small category $I$, an $I$-diagram $A$ of CDGAs, and an $A$-module $M$ in $I$-diagrams, we can construct a DGAA $\sD_A(M)$ of differential operators of $M$. When  $M$ is a strict line bundle, $\sD_A(M)$ then gives rise to a filtered DGLA $Q\widehat{\Pol}(M,-1)$ governing $E_0$ quantisations of the diagram $M$. 

If we could choose appropriate restrictions on $(A,M)$ to ensure that $\sD_A(M)$ had the correct  homotopical properties, passage to Maurer--Cartan spaces would then naturally give a presentation of the $\infty$-category of $E_0$-quantisations as a derivator. 
However, this is not straightforward, since in order to compute both de Rham and Poisson cohomology correctly, we need the $A$-modules $\Omega^m_A$ to be both fibrant and cofibrant.

When $I$ is a category of  the form $[m]= (0 \to 1 \to \ldots \to m)$, and $A$ is fibrant and cofibrant in the injective model structure, this condition is satisfied, so we can construct Maurer--Cartan spaces of $[m]$-diagrams, providing all the data necessary  to form Segal spaces.

\subsection{Quantised polyvectors  for diagrams}\label{DMdiagsn}

We now construct differential operators and quantised polyvectors for suitable diagrams of derived affine schemes.

\begin{definition} \label{DMDdef}
Given a small category $I$, an $I$-diagram $A$ of $R$-CDGAs, and   $A$-modules $M,N$ in $I$-diagrams of cochain complexes, define the filtered cochain complex $\Diff(M,N)= \Diff_{A/R}(M,N)\subset \HHom_R(M,N)$ of differential operators from $M$ to $N$ as the equaliser of the obvious diagram
\[
\prod_{i\in I} \Diff_{A(i)/R}(M(i),N(i)) \implies \prod_{f\co i \to j \text{ in } I}   \Diff_{A(i)/R}(M(i),f_*N(j)),
\]
with the filtration $F_k \Diff(M,N)$ defined similarly. 

We then write $\sD(M)= \sD_{A/R}(M):= \Diff_{A/R}(M,M)$,   which we regard as a  DGAA under  composition. We simply write $\sD_A= \sD_{A/R}$ for  $\sD_{A/R}(A,A)$. 
\end{definition}

For $f \co i \to j$ a morphism in $I$, the maps 
\[
 \gr^F_{k} \Diff_{A(i)/R}(M(i),f_*N(j)) \to \HHom_{A(i)}(M(i)\ten_{A(i)}\CoS^k_{A(i)/R}\Omega^1_{A(i)},f_*N(j))
\]
are isomorphisms whenever $A(i)$ is semi-smooth and $M(i)^{\sharp}$ projective over $A(i)^{\sharp}$. When these conditions hold for all $i$, the maps
\[
 \gr^F_{k} \Diff_{A/R}(M,N) \to \HHom_A(M\ten_A\CoS^k_A\Omega^1_{A/R}, N) 
\]
are thus also isomorphisms. 

We now have analogues of all the constructions in \S\S \ref{qpolsn}, \ref{centresn}.

\begin{definition}\label{Iqpoldef}
Given an $I$-diagram $A$ of $R$-CDGAs, and an $I$-diagram $M$ of strict line bundles over $A$, define the filtered DGLA of quantised $(-1)$-shifted polyvector fields  on $M$ by
\begin{align*}
 Q\widehat{\Pol}(M,-1)&:= \prod_{j \ge 0}F_j\sD_{A/R}(M)\hbar^{j-1}\\
\tilde{F}^iQ\widehat{\Pol}(M,-1)&:= \prod_{j \ge i} F_j\sD_{A/R}(M)\hbar^{j-1}\\
G^kQ\widehat{\Pol}(M,-1)&:=\hbar^kQ\widehat{\Pol}(M,-1).
\end{align*}
\end{definition}

\begin{definition}\label{IQPdef}
 We then define $Q\cP(M,-1)$, $TQ\widehat{\Pol}(M,-1)$, $T_{\Delta}Q\widehat{\Pol}(M,-1)$, 
$TQ\cP(M,-1)$, $T_{\Delta}Q\cP(M,-1)$,
  $\sigma=-\pd_{\hbar^{-1}}\co Q\widehat{\Pol}(M,-1) \to TQ\widehat{\Pol}(M,-1)$ and  $\sigma \co Q\cP(M,-1)\to TQ\cP(M,-1)$ as before, replacing Definition \ref{qpoldef} with Definition \ref{Iqpoldef}. 
\end{definition}

Note that if $u \co I \to J$ is a morphism of small categories and $A$ is a functor from  $J$ to $R$-CDGAs with $B= A \circ u$, then we have natural maps $F(A) \to F(B)$ for all the constructions $F$ of Definition \ref{IQPdef}.

The following is \cite[Lemma \ref{poisson-calcTOmegalemma}]{poisson}:
\begin{lemma}\label{calcTOmegalemma}
 If $A$ is  $[n]$-diagram in   $R$-CDGAs which is cofibrant and  fibrant for the injective model structure (i.e. each $A(i)$ is cofibrant and the maps $A(i) \to A(i+1)$ are surjective), then $\HHom_A(\CoS_A^k\Omega^1_A,A)$ is a model for the derived $\Hom$-complex $\oR\HHom_A(\oL\CoS_A^k\oL\Omega^1_A,A)$, and $ \HHom_A(A,\Omega^m_A) \simeq \ho\Lim_i \oL\Omega^m_{A(i)}$.
\end{lemma}
%

When $A$ satisfies  the conditions of Lemma \ref{calcTOmegalemma}, the lemma  combines with the observation above to show that every strict line bundle $M$ over $A$ satisfies
\[
 \gr^F_{k} \sD_{A/R}(M) \simeq \oR\HHom_A(\oL\CoS^k_A\oL\Omega^1_{A/R}, A). 
\]

\begin{definition}\label{ICompdef}
Given an $[m]$-diagram $A$ satisfying the conditions of Lemma \ref{calcTOmegalemma},   define  
\[
 G\PreSp(A/R,-1):=  G\PreSp(A(0)/R,-1)= \Lim_{i\in [m]}  G\PreSp(A(i)/R,-1),
\]
for the space $G\PreSp$ of generalised pre-symplectic structures of Definition \ref{GPreSpdef}. 

For a strict line bundle $M$ over $A$, define 
\[
 \mu \co G\PreSp(A/R,-1) \by Q\cP(M,-1) \to TQ^{tw}\cP(M,-1) 
\]
by setting $\mu(\omega, \Delta)(i):= \mu(\omega(i), \Delta(i)) \in TQ\cP(M(i),-1)$ for $i \in [m]$, and let $ Q\Comp(M,-1)$ be the homotopy vanishing locus of
\[
(\mu - \sigma) \co  G\PreSp(A/R,-1) \by Q\cP(M,-1) \to  TQ^{tw}\cP(M,-1).
\]
over $Q^{tw}\cP(M,-1)$.
\end{definition}

The following is  \cite[Lemma \ref{calcTlemma2}]{poisson}:
\begin{lemma}\label{calcTlemma2}
If $D=(A\to B)$ is a fibrant cofibrant  $[1]$-diagram of $R$-CDGAs which is formally \'etale in the sense that the map
$
 \Omega_{A}^1\ten_{A}B \to \Omega_{B}^1
$
is a quasi-isomorphism, then the map    
\[
 \HHom_D(\CoS_D^k\Omega^1_D,D) \to \HHom_{A}(\CoS_{A}^k\Omega^1_{A},A),
\]
is a quasi-isomorphism.
\end{lemma}

\subsection{Towers of obstructions}\label{obsn}

We now show how to adapt the various obstruction towers from \S \ref{affinesn} to apply to $[m]$-diagrams of derived affines.

\begin{definition}
 For  an $[m]$-diagram $A$ and $k \ge 1$, define 
\begin{align*}
  \Ob( Q\cP,A,k)&:= \mmc( F^2\widehat{\Pol}(A,-1) \oplus \hbar^k F^{2-k}\widehat{\Pol}(A,-1)[1]),\\
\Ob( Q^{tw}\cP,A,k)&:= \mmc( F^2\widehat{\Pol}(A,-1) \oplus \hbar^k F^{2-2k}\widehat{\Pol}(A,-1)[1]),
\end{align*}
where the DGLA structure is defined by regarding the second term as a module over the first. Note that these expressions only differ for $k=1$, as $F^{<0}=F^0$.
\end{definition}

Projection gives a fibration $\Ob(Q\cP,A,k) \to \cP(A,-1)=  Q\cP(A,-1)/G^1$, with the fibre over $\pi$ being $ \mmc(\hbar^{k} F^{2-k}T_{\pi} \widehat{\Pol}(A,-1)[1])$, and similarly for $\Ob(Q^{tw}\cP,A,k)$.

For strict line bundles $M$ over $I$-diagrams $A$, the  extension $\tilde{F}^2Q\widehat{\Pol}(M,-1)/G^{k+1} \to \tilde{F}^2Q\widehat{\Pol}(M,-1)/G^k$ of DGLAs, with abelian kernel $ \hbar^k F^{2-k}\widehat{\Pol}(A,-1)$, and its analogue for $(G*\tilde{F})^2$ give rise to canonical fibration sequences
\begin{align*}
 Q\cP(M,-1)/G^{k+1} \to Q\cP(M,-1)/G^k &\xra{\ob} \Ob( Q\cP,A,k),\\
Q^{tw}\cP(M,-1)/G^{k+1} \to Q^{tw}\cP(M,-1)/G^k &\xra{\ob} \Ob( Q^{tw}\cP,A,k)
\end{align*}
over $\cP(A,-1) $.

Similarly, we have a fibration sequence
\[
 G\PreSp(A/R,-1)/G^{k+1} \to G\PreSp(A/R,-1)/G^k \xra{\ob} \mmc(\hbar^k \DR(A(0)/R)[1]).
\]

We also have a map $\sigma = -\pd_{\hbar^{-1}} \co \Ob( Q\cP,A,k) \to \hbar \Ob( Q\cP,A,k)$, and  maps
\begin{align*}
  \PreSp(A/R,-1)\by\Ob( Q\cP,A,k)  \xra{\nu} \hbar \Ob( Q^{tw}\cP,A,k)\\
   \mmc(\hbar^k \DR(A(0)/R)[1])\by\cP(A,-1) \xra{\mu} \hbar \Ob( Q^{tw}\cP,A,k)
\end{align*}
given by $\nu(\omega, \pi, u)(i)= (\pi(i),\nu(\omega(i), \pi(i), u(i)))$ and $\mu(v, \pi)(i)= (\pi(i), \mu(v(i), \pi(i)))$ for $i \in [m]$. 

\begin{definition}\label{ObCompPdef}
 For $k \ge 1$, define $\Ob(Q\Comp/Q\cP, A,k)$ to be the homotopy vanishing locus of 
\[
\mu \co \mmc(\hbar^k \DR(A(0)/R)[1])\by\cP(A,-1) \to  \hbar \Ob( Q^{tw}\cP,A,k).
\]
over $\cP(A,-1)$
\end{definition}

Combining the earlier  fibration sequences with the definition of $Q\Comp$, we have
\begin{lemma}\label{compSpob}
 There is a natural obstruction map
\[
\ob \co (Q\Comp(M,-1)/G^k)\by^h_{ (Q\cP(M,-1)/G^k)} (Q\cP(M,-1)/G^{k+1})\to \Ob(Q\Comp/Q\cP, A,k)
\]
over $\cP(A,-1) $, whose homotopy vanishing locus is
\[
 Q\Comp(M,-1)/G^{k+1}.
\]
\end{lemma}

\begin{definition}\label{ObCompSdef}
 Define $\Ob(Q\Comp/Q\cS, A,k)$ to be the homotopy vanishing locus of 
\[
 (\nu - \sigma\circ \pr_2) \co  \Comp(A,-1)\by_{\cP(A,-1)}\Ob( Q\cP,A,k)\to \hbar \Ob( Q^{tw}\cP,A,k)
\]
over $\cP(A,-1) $.
\end{definition}

Combining the earlier  fibration sequences with the definition of $Q\Comp$, we also have:
\begin{lemma}\label{compPob}
 There is a natural obstruction map
\[
\ob \co (Q\Comp(M,-1)/G^k)\by^h_{ (G\PreSp(A,-1)/G^k)} (G\PreSp(A,-1)/G^{k+1})\to \Ob(Q\Comp/Q\cS, A,k)
\]
over $\Comp(A,-1) $, whose homotopy vanishing locus is
\[
 Q\Comp(M,-1)/G^{k+1}.
\]
\end{lemma}

\subsection{Descent and line bundles}\label{DMdescentsn}

In order to define our various structures on derived DM $N$-stacks, and to look at quantisations of line bundles, we now make use of \'etale descent and functoriality.


\begin{definition}
 Write $dg\CAlg(R)$ for the category of CDGAs over $R$, and let $dg\CAlg(R)_{c, \onto}\subset dg\CAlg(R) $ be the subcategory with all cofibrant $R$-CDGAs as objects, and only surjective morphisms.
\end{definition}

We already have simplicial set-valued functors $G\PreSp(-,-1)$ and $G\Sp(-,-1)$ from  $dg\CAlg(R)$ to $s\Set$, mapping quasi-isomorphisms in $dg\CAlg(R)_{c}$ to weak equivalences. Poisson structures and their quantisations are only functorial with respect to formally \'etale morphisms, in an $\infty$-functorial sense which we now make precise. 

Observe that,  when $F$ is any  of the constructions $Q\cP(-,-1)$, $Q\Comp(-,-1)$, $G\PreSp(-,-1)$, $\Ob( Q\cP,-,k)$, $\Ob(Q\Comp/Q\cP, -,k)$ or $\Ob(Q\Comp/Q\cS, -,k) $ applied to $[m]$-diagrams in $dg\CAlg(R)_{c, \onto}$, Lemmas \ref{calcTOmegalemma} and \ref{calcTlemma2} combine with  the obstruction calculus of \S \ref{obsn} to show that \cite[Properties \ref{poisson-Fproperties}]{poisson} are satisfied:
\begin{properties}\label{Fproperties}
\begin{enumerate}
 \item the maps from $F(A(0)\to \ldots \to A(m))$ to 
\[
 F(A(0)\to A(1))\by_{F(A(1)}^h F(A(1)\to A(2))\by^h_{F(A(2)}\ldots \by_{F(A(n-1)}^hF(A(n-1) \to A(n))
\]
 are weak equivalences;

\item if the $[1]$-diagram $A \to B$ is a quasi-isomorphism, then the natural maps from $F(A \to B)$ to $F(A)$ and to $F(B)$ are weak equivalences. 

\item if the $[1]$-diagram $A \to B$ is formally \'etale, then the natural map from $F(A \to B)$ to $F(A)$ is a  weak equivalence.
\end{enumerate}
\end{properties}

The first two properties ensure that the simplicial classes $\coprod_{ A \in B_mdg\CAlg(R)_{c, \onto}} F(A)$ fit together to give a complete Segal space $\int F$ over the nerve $Bdg\CAlg(R)_{c, \onto} $. Taking Segal spaces as our preferred model of $\infty$-categories, we define $\oL dg\CAlg(R)_{c, \onto}$ and  $\oL dg\CAlg(R)$ to be the $\infty$-categories obtained by localising the respective ($\infty$-)categories at quasi-isomorphisms or weak equivalences, and let  $\oL dg\CAlg(R)^{\et} \subset\oL dg\CAlg(R) $  the be the $\infty$-subcategory of homotopy formally \'etale morphisms. 

For any construction $F$ satisfying the conditions above,  \cite[Definition \ref{poisson-inftyFdef}]{poisson} then gives
an $\infty$-functor
\[
 \oR F \co \oL dg\CAlg(R)^{\et} \to  \oL s\Set
\]
to the $\infty$-category of simplicial sets, with the property that 
\[
 (\oR F)(A) \simeq F(A)
\]
for all cofibrant $R$-CDGAs $A$.


\begin{definition}\label{DMFdef}
Given a derived Deligne--Mumford $N$-stack $\fX$ and any of the constructions $F$ above, define
$
  F(\fX)
$
 to be the homotopy limit of $\oR F(A)$ over the $\infty$-category $(DG\Aff_{\et}\da \fX)$ consisting of  derived affines  $\Spec A$ equipped with homotopy \'etale (i.e. \'etale in the sense of \cite{hag2}) maps to  $\fX$, and all homotopy \'etale morphisms between them.  
\end{definition}

When $\fX\simeq  \Spec B$ is a derived affine, note that it is final in the category of derived affines over $\fX$, so $\oR F(\fX)= \oR F(B)= F(B)$. In general, it suffices to take the homotopy limit over any subcategory of  $(DG\Aff_{\et}\da \fX)$ with colimit $\fX$, so this definition also coincides with \cite[Definition \ref{poisson-DMFdef}]{poisson}, by applying it to a suitable hypergroupoid. 

Definition \ref{DMFdef} is insufficient for our purposes, as we wish to consider line bundles. Since DM stacks only involve CDGAs with non-positive cohomology, the line bundles we encounter will be locally trivial, so for now we only need to set up $\bG_m$-equivariance. 

\begin{definition}
 Define the functor $\bG_m$ from CDGAs to groups by 
\begin{align*}
 \bG_m(A) &:= \z^0(A)^{\by}.
\end{align*}

Given a string $A=(A(0) \to \ldots \to A(n))$ of CDGAs, we write $\bG_m(A):= \bG_m(A(0))$, regarded as $\Lim_i \bG_m(A(i))$.
\end{definition}

Now, the group $\bG_m(A)^{\by} $ acts by conjugation on $Q\widehat{\Pol}(A,-1) $, corresponding to automorphisms of $A$ as a line bundle over $A$.  This action  preserves all the filtrations, so it acts on the simplicial set  $ Q\cP(A,-1) = \mmc(\tilde{F}^2Q\widehat{\Pol}(A,-1))$ respecting the cofiltrations.  Note that the action is trivial on the quotient $Q\cP(A,-1)/G^1= \cP(A,-1)$. 

\begin{definition}
For any of the constructions $F$ above, let $\oR (F/^h\bG_m)$ be the $\infty$-functor on $\oL dg\CAlg(R)^{\et}$  given by applying the construction of  \cite[Definition \ref{poisson-DMFdef}]{poisson} to the homotopy quotient $F/^h\bG_m$, then taking \'etale hypersheafification. 
\end{definition}

Up to now, hypersheafification has not been necessary because all our functors have been hypersheaves --- this follows because the associated gradeds $\gr_{F}$ of obstruction functors can be written in terms of tangent sheaves and sheaves of differential forms. However, $B\bG_m$ requires hypersheafification because  the simplicial presheaf $B\bG_m$ does not preserve weak equivalences or satisfy \'etale descent.

For any derived line bundle $\sL$ on a derived stack $\fX$, there is an associated $\bG_m$-torsor given locally by the disjoint union of spaces of quasi-isomorphisms from $\O_{\fX}[m]$ to $\sL$ for all $m \in \Z$. 

\begin{definition}
 Given  a derived Deligne--Mumford $N$-stack $\fX$, a derived line bundle $\sL$ on $\fX$ and any of the constructions $F$ above, define
$
  F(\sL)
$
 to be the homotopy limit of $\oR(F/^h\bG_m)(A)\by_{\oR (*/^h\bG_m)(A)}^h\{\sL|_{A}\}$ over the $\infty$-category $(DG\Aff_{\et}\da \fX)$.
\end{definition}

\begin{remarks}
 If we fix $\pi \in \cP(\fX,-1)$, then observe that the homotopy fibres of  $ Q\cP(\sL,-1)\to Q\cP(\sL,-1)/G^1= \cP(M,-1)$ over $\pi$ can be combined and enhanced to a   dg category whose objects are $E_0$ quantisations $(\sL, \Delta)$ over $\pi$, with dg morphisms $(\sL,\Delta) \to (\sL',\Delta')$ given by the complex  $(\prod_{j \ge 0} F_{j}\mathscr{D}\!\mathit{iff}_{\fX/R}(\sL,\sL')\hbar^{j},(\delta + \Delta')_* \mp f(\delta + \Delta)^*)$ and the obvious composition law.

Also note that the action of $\bG_m(A)^{\by} $ is unipotent, as it is trivial on $\cP(A,-1)$, so it extends naturally to an action of $\bG_m(A)^{\by}\ten_{\Z}\Q$, and we can therefore define quantisations for \'etale $B(\bG_m/\mu_{\infty})$-torsors. In fact, more is true: line bundles $\sM$ with left $\sD$-module structure give equivalences $\sD(\sL) \simeq \sD(\sL \ten \sM)$ and hence $ Q\cP(\sL,-1) \simeq  Q\cP(\sL\ten \sM,-1)$. By sheafifying these equivalences, we get a notion of quantisation for all elements of $\mmc(F^1\DR(\fX/R)[1])$ via Chern classes.
\end{remarks}

\subsection{Comparing quantisations and generalised symplectic structures}\label{compDMsn}

We now fix a strongly quasi-compact derived  DM $N$-stack $\fX$ over $R$.

\begin{lemma}\label{CompPObX}
 For $(\omega, \pi) \in \Comp(\fX,-1)$, the homotopy fibre 
\[
\Ob( Q\Comp/Q\cS,\fX,k)_{(\omega, \pi)} \quad\text{ of }\quad
 \Ob( Q\Comp/Q\cS,\fX,k) \to \Comp(\fX,-1)
\]
over $(\omega,\pi)$ is contractible for all $k \ge 2$.
\end{lemma}
\begin{proof}
As in the proof of Lemma \ref{tangentlemma}, the map $\hbar^{-1}\nu(\omega, \pi,-)- (p-1)$ on 
  \[
 \hbar^{p-1} \Ext^{2-i}_{\sO_{\fX}}(\oL\CoS_{\sO_{\fX}}^{p-k}\oL\Omega^1_{\fX/R},\sO_{\fX}) 
\]
is invertible for $k \ge 2$, so Lemma \ref{compPob} gives contractibility of  the homotopy fibre. 
\end{proof}

\begin{proposition}\label{quantprop2}
For any line bundle $\sL$ on $\fX$, the  map
\begin{align*}
 Q\Comp(\sL,-1) &\to (Q\Comp(\sL,-1)/G^2) \by^h_{(G\PreSp(\fX,-1)/G^2)} G\PreSp(\fX,-1)\\
&\simeq (Q\Comp(\sL,-1)/G^2)\by \prod_{i \ge 2} \mmc(\DR(\fX/R)\hbar^i).
\end{align*}
is a weak equivalence.
\end{proposition}
\begin{proof}
This is essentially the same as the  proof of Proposition \ref{quantprop}. Lemma \ref{CompPObX} combines with the obstruction maps
\[
 (Q\Comp(\sL,-1)/G^k) \by^h_{(G\PreSp(\fX,-1)/G^k)}(G\PreSp(\fX,-1)/G^{k+1}) \to \Ob(Q\Comp/Q\cS,\fX, k)
\]
to give the weak equivalences 
\[
(Q\Comp(\sL,-1)/G^{k+1}) \simeq (Q\Comp(\sL,-1) /G^k)\by^h_{(G\PreSp(\fX,-1)/G^k)} (G\PreSp(\fX,-1)/G^{k+1}). 
\]
\end{proof}

\begin{definition}\label{nondegstack}
 Given a $(-1)$-shifted Poisson structure $\pi \in \cP(\fX,-1)$, we say that $\pi$ is non-degenerate if the induced map
\[
 \pi^{\sharp} \co \oL\Omega^1_{\fX} \to \oR\cHom_{\sO_X}(\oL\Omega^1_{\fX}, \sO_{\fX})[1]
\]
is a quasi-isomorphism of sheaves on $\fX$, and $\oL\Omega^1_{\fX} $ is perfect.
\end{definition}

\begin{lemma}\label{CompSpObX}
 If $\pi$ is a non-degenerate $(-1)$-shifted Poisson structure on $\fX$, then the homotopy fibre $\Ob( Q\Comp/Q\cP, \fX, k)_{\pi}$
of 
\[
 \Ob(Q\Comp/Q\cP, \fX, k) \to \cP(\fX,-1)
\]
over $\pi$ is contractible.
\end{lemma}
\begin{proof}
 The map
\[
 \mu(-,\pi)\co \H^{2-i}(\fX, \oL\Omega^{p-2k}_{\fX/R})\to  \hbar^p \Ext^{2-i}_{\sO_{\fX}}(\oL\CoS_{\sO_{\fX}}^{p-2k}\oL\Omega^1_{\fX/R},\sO_{\fX})
\]
is given by $\Lambda^{p-2k}\pi^{\sharp}$, so is an isomorphism by the non-degeneracy of $\pi$. Lemma \ref{compSpob} then gives contractibility of $\Ob( Q\Comp/Q\cP, \fX, k)_{\pi}$.
\end{proof}

\begin{proposition}\label{QcompatP2}
 For any  line bundle $\sL$ on $\fX$, the canonical map
\begin{eqnarray*}
    Q\Comp(\sL,-1)^{\nondeg} \to  Q\cP(\sL,-1)^{\nondeg}           
\end{eqnarray*}
 is a weak equivalence. In particular, there is a morphism
\[
  Q\cP(\sL,-1)^{\nondeg} \to G\Sp(\fX,-1)
\]
in the homotopy category of simplicial sets.
\end{proposition}
\begin{proof}
 This is much the same as Proposition \ref{QcompatP1}.  Lemma \ref{CompSpObX} combines with the obstruction maps
\[
 (Q\Comp(\sL,-1)/G^k) \by^h_{(Q\cP(\sL,-1)/G^k)}( Q\cP(\sL,-1)/G^{k+1}) \to \Ob(Q\Comp/Q\cP,\fX, k)
\]
to give the weak equivalences 
\begin{align*}
 &(Q\Comp(\sL,-1)^{\nondeg}/G^{k+1}) \simeq\\
 &(Q\Comp(\sL,-1) /G^k)\by^h_{(Q\cP(\sL,-1)/G^k)}( Q\cP(\sL,-1)^{\nondeg}/G^{k+1}).
\end{align*}
\end{proof}

\section{Quantisation for derived Artin $N$-stacks}\label{Artinsn}

In order to proceed further, we will make use of the \'etale resolutions of derived Artin stacks by stacky CDGAs given in  \cite[\S \ref{poisson-Artinsn}]{poisson}. We just extend the results of \S \ref{affinesn} from CDGAs to stacky CDGAs, and then the \'etale descent approach of \S \ref{DMsn} adapts immediately.

\subsection{Stacky CDGAs}\label{bicdgasn}
We now recall some definitions and lemmas from \cite[\S \ref{poisson-Artinsn}]{poisson}.
From now on, we will  regard the CDGAs encountered so far  as chain complexes $\ldots \xra{\delta} A_1 \xra{\delta} A_0 \xra{\delta} \ldots$ rather than cochain complexes --- this will enable us to distinguish easily between derived (chain) and stacky (cochain) structures.

\begin{definition}
A stacky CDGA is  a chain cochain complex $A^{\bt}_{\bt}$ equipped with a commutative product $A\ten A \to A$ and unit $k \to A$.  Given a chain  CDGA $R$, a stacky CDGA over $R$ is then a morphism $R \to A$ of stacky CDGAs. We write $DGdg\CAlg(R)$ for the category of  stacky CDGAs over $R$, and $DG^+dg\CAlg(R)$ for the full subcategory consisting of objects $A$ concentrated in non-negative cochain degrees.
\end{definition}

When working with chain cochain complexes $V^{\bt}_{\bt}$, we will usually denote the chain differential by $\delta \co V^i_j \to V^i_{j-1}$, and the cochain differential by $\pd \co V^i_j \to V^{i+1}_j$.

\begin{definition}
 Say that a morphism $U \to V$ of chain cochain complexes is a levelwise quasi-isomorphism if $U^i \to V^i$ is a quasi-isomorphism for all $i \in \Z$. Say that a morphism of stacky CDGAs is a levelwise quasi-isomorphism if the underlying morphism of chain cochain complexes is so.
\end{definition}

The following is \cite[Lemma \ref{poisson-bicdgamodel}]{poisson}:
\begin{lemma}\label{bicdgamodel}
There is a cofibrantly generated model structure on stacky CDGAs over $R$ in which fibrations are surjections and weak equivalences are levelwise quasi-isomorphisms. 
\end{lemma}

There is a denormalisation functor $D$ from non-negatively graded CDGAs to cosimplicial algebras, with 
 left adjoint $D^*$ as in \cite[Definition \ref{ddt1-nabla}]{ddt1}. 
Given a cosimplicial chain CDGA $A$, $D^*A$ is then a stacky CDGA in non-negative cochain degrees. By  \cite[Lemma \ref{poisson-Dstarlemma}]{poisson}, $D^*$ is a left Quillen functor from the Reedy model structure on cosimplicial chain CDGAs to the model structure of Lemma \ref{bicdgamodel}.

Since   $DA$ is a pro-nilpotent extension of $A^0$, when $\H_{<0}(A)=0$ we think of the hypersheaf  $\oR \Spec DA$ as a stacky derived thickening of the derived affine scheme $\oR \Spec A^0$.  

\begin{definition}
 Given a chain cochain complex $V$, define the cochain complex $\hat{\Tot} V \subset \Tot^{\Pi}V$ by
\[
(\hat{\Tot} V)^m := (\bigoplus_{i < 0} V^i_{i-m}) \oplus (\prod_{i\ge 0}   V^i_{i-m})
\]
with differential $\pd \pm \delta$.
\end{definition}

\begin{definition}
 Given a stacky CDGA $A$ and $A$-modules $M,N$ in chain cochain complexes, we define  internal $\Hom$s
$\cHom_A(M,N)$  by
\[
 \cHom_A(M,N)^i_j=  \Hom_{A^{\#}_{\#}}(M^{\#}_{\#},N^{\#[i]}_{\#[j]}),
\]
with differentials  $\pd f:= \pd_N \circ f \pm f \circ \pd_M$ and  $\delta f:= \delta_N \circ f \pm f \circ \delta_M$,
where $V^{\#}_{\#}$ denotes the bigraded vector space underlying a chain cochain complex $V$. 

We then define the  $\Hom$ complex $\hat{\HHom}_A(M,N)$ by
\[
 \hat{\HHom}_A(M,N):= \hat{\Tot} \cHom_A(M,N).
\]
\end{definition}
Note that  there is a multiplication $\hat{\HHom}_A(M,N)\ten \hat{\HHom}_A(N,P)\to \hat{\HHom}_A(M,P)$.

\begin{definition}\label{hfetdef}
 A morphism  $A \to B$ in $DG^+dg\CAlg(R)$ is said to be  homotopy formally \'etale when the map
\[
 \{\Tot \sigma^{\le q} (\oL\Omega_{A}^1\ten_{A}^{\oL}B^0)\}_q \to \{\Tot \sigma^{\le q}(\oL\Omega_{B}^1\ten_B^{\oL}B^0)\}_q
\]
on the systems of brutal cotruncations is a pro-quasi-isomorphism.
\end{definition}

Combining \cite[Proposition \ref{poisson-replaceprop}]{poisson} with  \cite[Theorem \ref{stacks2-bigthm} and Corollary \ref{stacks2-Dequivcor}]{stacks2}, every strongly quasi-compact derived Artin $N$-stack over $R$ can be resolved by a homotopy formally \'etale cosimplicial diagram in $DG^+dg\CAlg(R)$.

\subsection{Quantised polyvectors}\label{biquantsn}
We now fix a  stacky CDGA $A$ over a chain CDGA $R$.

\begin{definition}
Given  $A$-modules $M,N$ in chain  cochain complexes, inductively define the 
filtered chain cochain complex $\cDiff(M,N)= \cDiff_{A/R}(M,N)\subset \cHom_R(M,N)$ of differential operators from $M$ to $N$  by setting
\begin{enumerate}
 \item $F_0 \cDiff(M,N)= \cHom_A(M,N)$,
\item $F_{k+1} \cDiff(M,N)=\{ u \in \cHom_R(M,N)~:~  [a,u]\in F_{k} \cDiff(M,N)\, \forall a \in A \}$, where $[a,u]= au- (-1)^{\deg a\deg u} ua$.
\item $\cDiff(M,N)= \LLim_k F_k\cDiff(M,N)$.
\end{enumerate}

We then define the filtered cochain complex $\hat{\Diff}(M,N)= \hat{\Diff}_{A/R}(M,N)\subset \hat{\HHom}_R(M,N)$ by $\hat{\Diff}(M,N):= \LLim_k \hat{\Tot} F_k\cDiff(M,N)$.
\end{definition}
%

\begin{definition}
Given  an $A$-module $M$ in chain cochain complexes, write $\sD(M)= \sD_{A/R}(M):= \hat{\Diff}_{A/R}(M,M)$,   which we regard as a  sub-DGAA of $\hat{\HHom}_R(M,M)$. We simply write $\sD_A= \sD_{A/R}$ for  $\sD_{A/R}(A,A)$. 
\end{definition}

\begin{definition}\label{bistrictlb}
 Define a strict line bundle over $A$ to be an $A$-module $M$ in chain cochain complexes such that $M^{\#}_{\#}$ is a projective module of rank $1$ over the bigraded-commutative algebra $A^{\#}_{\#}$ underlying $A$. 
\end{definition}

Definitions \ref{qpoldef} and \ref{QFdef}  then carry over verbatim to   define  quantised polyvectors over a stacky CDGA, and the filtrations $\tilde{F}$, $G$, and $G*\tilde{F}$.

We now follow \cite[\S \ref{poisson-bipoisssn}]{poisson} in making the following assumptions on $A \in DG^+dg\CAlg(R)$:
\begin{enumerate}
 \item for any cofibrant replacement $\tilde{A}\to A$ in the model structure of Lemma \ref{bicdgamodel}, the morphism $\Omega^1_{\tilde{A}/R}\to \Omega^1_{A/R}$ is a levelwise quasi-isomorphism,
\item  the $A^{\#}$-module $(\Omega^1_{A/R})^{\#}$  in graded chain complexes is cofibrant (i.e. it has the left lifting property with respect to all surjections of $A^{\#}$-modules in graded chain complexes),
\item there exists $N$ for which the chain complexes $(\Omega^1_{A/R}\ten_AA^0)^i $ are acyclic for all $i >N$.
\end{enumerate}
These conditions are satisfied by $D^*O(X)$ for DG Artin hypergroupoids $X$.

For stacky CDGAs of this form, Definition \ref{Qpoissdef} adapts verbatim to define the space $Q\cP(M,-1)$ of $E_0$ quantisations of a strict line bundle $M$, and its twisted counterpart $Q^{tw}\cP(M,-1) $. 
The second assumption gives us isomorphisms
\begin{align*}
  F^{p-2i}\widehat{\Pol}(A,-1)\hbar^{i} &\to \gr_G^i(G*\tilde{F})^pQ\widehat{\Pol}(M,-1),\\
T_{\pi_{\Delta}}F^{p-2i}\widehat{\Pol}(A,-1)\hbar^{i} &\to \gr_G^i(G*\tilde{F})^pT_{\Delta}Q\widehat{\Pol}(M,-1).
\end{align*}

The following is a slight generalisation of \cite[Definition \ref{poisson-binondegdef}]{poisson}:
\begin{definition}\label{biQnonnegdef}
 Say that an $E_0$ quantisation $\Delta =  \sum_{j \ge 2} \Delta_j \hbar^{j}$ of a strict line bundle $M$ over $A$ is non-degenerate if the map 
\[
\Delta_2^{\sharp}\co  \Tot^{\Pi}(M^0\ten_A \Omega^1_A) \to \hat{\HHom}_A(\Omega^1_A,M^0)[1]
\]
of is a quasi-isomorphism, and $\Tot^{\Pi}(\Omega_A^1\ten_AA^0)$ is a perfect complex over $A^0$.
\end{definition}

Definitions \ref{TQPdef} and \ref{Qsigmadef} adapt verbatim, giving tangent spaces
 $TQ\cP(M,-1)$, $TQ\cP(M,-1)/G^k$, $TQ^{tw}\cP(M,-1)$, $TQ^{tw}\cP(M,-1)/G^k$and a  canonical tangent vector
\[
  \sigma=-\pd_{\hbar^{-1}}\co Q\widehat{\Pol}(M,-1) \to TQ\widehat{\Pol}(M,-1).
\]

\subsection{Generalised symplectic structures and compatible quantisations}

The following is \cite[Definition \ref{poisson-biDRdef}]{poisson}
\begin{definition}\label{biDRdef}
Define the de Rham complex $\DR(A/R)$ to be the product total complex of the bicomplex
\[
 \Tot^{\Pi} A \xra{d} \Tot^{\Pi}\Omega^1_{A/R} \xra{d} \Tot^{\Pi}\Omega^2_{A/R}\xra{d} \ldots,
\]
so the total differential is $d \pm \delta \pm \pd$.

We define the Hodge filtration $F$ on  $\DR(A/R)$ by setting $F^p\DR(A/R) \subset \DR(A/R)$ to consist of terms $\Tot^{\Pi}\Omega^i_{A/R}$ with $i \ge p$.
\end{definition}

We then similarly define $\DR'(A/R)$ to be the  (triple) product total complex
\[
 \DR'(A/R):=\Tot^{\Pi} N\hat{A}^{\bt +1}
\]
regarded as a filtered DGAA over $R$, with $F^p\DR'(A/R):=\Tot^{\Pi} NF^p\hat{A}^{\bt +1}$. Definition \ref{tildeFDRdef} then carries over to give a filtration $\tilde{F}$ on $DR'(A/R)\llbracket \hbar \rrbracket$.

Definition \ref{GPreSpdef} carries over to give a space $G\PreSp(A/R,-1)$ of generalised $(-1)$-shifted pre-symplectic structures on $A/R$. We say that a generalised pre-symplectic structure $\omega$ is symplectic if its leading term $\omega_0 \in \PreSp(A/R,-1)$ is symplectic in the sense of \cite[Definition \ref{poisson-biPreSpdef}]{poisson}; explicitly, this says that 
$\Tot^{\Pi} (\Omega_A^1\ten_AA^0)$ is a perfect complex over $A^0$ and the map
\[
[\omega_0]^{\sharp}\co \hat{\HHom}_A(\Omega^1_A, A^0)[-n]\to  \Tot^{\Pi} (\Omega_{A/R}^1\ten_AA^0) 
\]
is a quasi-isomorphism.  
We then let $\Sp(A/R,n) \subset \PreSp(A/R,n)$  consist of the symplectic structures --- this is a union of path-components.
 
Lemmas \ref{QPolmudef} and \ref{keylemma} then adapt to give  compatible maps
\[
 \mu(-,\Delta) \co \DR'(A/R)\llbracket\hbar\rrbracket/G^k \to T_{\Delta}Q\widehat{\Pol}(M,-1)/G^k
\]
respecting the filtrations $(G*\tilde{F})$.

We may now define the space $Q\Comp(A/R, n)$ of compatible quantisations as in Definition \ref{Qcompdef}, with Proposition \ref{QcompatP1} adapting to show that 
\[
 Q\Comp(M,-1)^{\nondeg} \to  Q\cP(M,-1)^{\nondeg}
\]
is a weak  equivalence and Proposition \ref{quantprop} adapting to show that 
the resulting  maps
\begin{align*}
 Q\cP(M,-1)^{\nondeg}/G^k &\to (Q\cP(M,-1)^{\nondeg}/G^2)\by^h_{(G\Sp(A,-1)/G^2)}(G\Sp(A,-1)/G^k) \\ 
&\simeq (Q\cP(M,-1)^{\nondeg}/G^2)\by \prod_{i \ge 2} \mmc(\DR(A/R)\hbar^i)
\end{align*}
  are weak equivalences for all $k \ge 2$.

\subsection{Diagrams and derived Artin stacks}

We now generalise the constructions of \S \ref{DMsn} to stacky CDGAs, allowing us to adapt arguments for DM stacks to apply to Artin stacks.

\subsubsection{Diagrams}\label{Artindiagramsn}

We may now proceed as in \cite[\S \ref{poisson-bidescentsn}]{poisson}. For any small category $I$, any $I$-diagram $A$ in $DG^+dg\CAlg(R)$, and $A$-modules $M,N$ in $I$-diagrams of chain cochain complexes, we define  $\cDiff_{A/R}(M,N)$ to be the equaliser of
\[
 \prod_{i \in I} \cDiff_{A(i)/R}(M(i),N(i)) \implies \prod_{f\co i \to j \text{ in } I}   \cDiff_{A(i)/R}(M(i),f_*N(j)).
\]

The constructions $Q\cP(-,-1)$, $G\PreSp(-,-1)$ and $Q\Comp(-,-1)$ all adapt to such diagrams, and behave well for $[m]$-diagrams $A$ which are fibrant and cofibrant for the injective $[m]$-diagram model structure on stacky CDGAs, so
 $A(i)$ is cofibrant for the model structure of Lemma \ref{bicdgamodel} and the maps $A(i) \to A(i+1)$ are all surjective. In particular, these constructions satisfy the conditions of \cite[\S \ref{poisson-bidescentsn}]{poisson}, so for each construction $F$ we have an $\infty$-functor
\[
 \oR F \co \oL DG^+dg\CAlg(R)^{\et} \to  \oL s\Set
\]
on the $\infty$-categories given by localising weak equivalences,
with $(\oR F)(A) \simeq F(A)$
for all cofibrant stacky CDGAs $A$ over $R$.  Here, $DG^+dg\CAlg(R)^{\et} \subset DG^+dg\CAlg(R)$ is the subcategory of morphisms $A \to B$ which are homotopy formally \'etale in the sense  of Definition \ref{hfetdef}.

By naturality of these constructions and the equivalences above, we then have weak equivalences of $\infty$-functors
\[
 \oR Q\Comp(-,-1)^{\nondeg} \to  \oR Q\cP(-,-1)^{\nondeg}
\]
and 
\begin{align*}
 \oR Q\cP(-,-1)^{\nondeg}/G^k &\to (\oR Q\cP(-,-1)^{\nondeg}/G^2)\by^h_{(\oR G\Sp(-,-1)/G^2)}(\oR G\Sp(-,-1)/G^k) 
\end{align*}
  for all $k \ge 2$.


The approach of \cite[\S \ref{poisson-bidescentsn}]{poisson} now applies immediately to associate  to any of the constructions above an $\infty$-functor on derived Artin $N$-stacks, with natural transformations and equivalences carrying over. However, this is not quite sufficient for our purposes, since we wish to consider quantisations of non-trivial line bundles.

\subsubsection{Descent and line bundles}\label{lbsn}

We say that a morphism $A \to B$ in $DG^+dg\CAlg(R)$ is a covering if $A^0 \to B^0$ is faithfully flat. In particular, this implies that 
\[                                                                                                                                                                  
\ho\LLim_i \oR\Spec D^iB \to \ho\LLim_i \oR\Spec D^iA                                                                                                   \]
 is a surjection of \'etale hypersheaves. 
 Note that when  $X \to Y$ is  a relative trivial derived Artin hypergroupoid, $X_0 \to Y_0$ is faithfully flat, so the morphism $D^*O(Y) \to D^*O(X)$ is a covering in the sense above. 

In \S \ref{DMdescentsn}, we were able to extend the functor $Q\cP$ to line bundles solely by making use of the $\bG_m$-action on it. For Artin stacks, the situation is more subtle, because for any $A \in DG^+\Alg(\Q)$, we have
$
 \Hom(D^*O(B\bG_m), A)  \cong \z^1A.
$

The most na\"ive simplicial set-valued functor we can consider on $DG^+dg\Alg(\Q)$ is $(B\bG_m)^{\Delta}\circ D$, which is represented by the cosimplicial CDGA $D^*O((B\bG_m)^{\Delta})$, and sends $A$ to the nerve
$
B[\z^1(\z_0A)/(\z_0A^0)^{\by}],
$ 
of the groupoid
\[
 \mathrm{TLB}(A):= [\z^1(\z_0A)/(\z_0A^0)^{\by}],
\]
where $f \in (A^0)^{\by}$ acts on $\z^1A$ by addition of $\pd \log f = f^{-1}\pd f$. We think of $\mathrm{TLB}(A)$  as the groupoid of trivial line bundles. 
 
For any cofibrant $A \in DG^+dg\Alg(R)$, we can extend $ Q\cP$   to a simplicial representation of the 
groupoid $\mathrm{TLB}(A)$ above by sending an object 
$b \in \z^1(\z_0A)$ to $Q\cP(A_b, -1)$, with $(\z_0A^0)^{\by}$ acting via functoriality for line bundles. Note that the quotient representation $Q\cP(-,-1)/G^1= \cP(-,-1)$ is trivial;  we also set $G\PreSp$ to be a trivial representation $b \mapsto G\PreSp(A)$.

\begin{definition}
For any of the constructions $F$ of \S \ref{Artindiagramsn}, let $\oR (F/^h\bG_m)$ be the $\infty$-functor on $\oL dg\CAlg(R)^{\et}$  given by applying the construction of  \cite[\S \ref{poisson-bidescentsn}]{poisson} to the right-derived functor of the Grothendieck construction
\[
 A \mapsto 
\holim_{\substack{ \lra \\ b \in \mathrm{TLB}(A)}} F(A_b),
\]
 then taking  hypersheafification with respect to homotopy formally  \'etale coverings.
\end{definition}

Given a derived Artin $N$-stack $\fX$, and $A \in DG^+dg\CAlg(R)$, we say that an element  $f \in \ho \Lim_i \fX(D^iA)$ is homotopy formally \'etale if the induced morphism 
\[
N_cf_0^*\bL_{\fX/R} \to \{ \Tot \sigma^{\le q} \oL\Omega^1_{A/R}\ten^{\oL}_AA^0\}_q
\]
 from \cite[\S \ref{poisson-Artintgtsn}]{poisson}  is a pro-quasi-isomorphism. We then write $(dg_+DG\Aff_{\et}\da \fX)$ for the $\infty$-category of homotopy formally \'etale elements $f \in \ho \Lim_i \fX(D^iA)$ with homotopy formally \'etale maps $A \to B$ between them.

\begin{definition}
 Given  a derived Artin $N$-stack $\fX$, a line bundle $\sL$ on $\fX$ and any of the functors $F$ above, define
$
  F(\sL)
$
 to be the homotopy limit of 
\[
\oR(F/^h\bG_m)(A)\by_{\oR (*/^h\bG_m)(A)}^h\{\sL|_{A}\}
\]
 over the $\infty$-category $(dg_+DG\Aff_{\et}\da \fX)$.
\end{definition}

If we now fix a derived Artin $N$-stack $\fX$, Definition \ref{nondegstack} carries over verbatim to give a notion of non-degeneracy for a $(-1)$-shifted Poisson structure $\pi \in \cP(\fX,-1)$, and Propositions \ref{quantprop2} and \ref{QcompatP2} readily adapt (substituting the relevant results from \cite[\S \ref{poisson-Artinsn}]{poisson}), giving

\begin{proposition}\label{prop3}
For any line bundle $\sL$ on $\fX$, the  canonical maps
\begin{align*} 
Q\Comp(\sL,-1)^{\nondeg} &\to  Q\cP(\sL,-1)^{\nondeg}\\ 
 Q\Comp(\sL,-1) &\to (Q\Comp(\sL,-1)/G^2) \by^h_{(G\PreSp(\fX,-1)/G^2)} G\PreSp(\fX,-1)\\
&\simeq (Q\Comp(\sL,-1)/G^2)\by \prod_{i \ge 2} \mmc(\DR(\fX/R)\hbar^i).
\end{align*}
are filtered weak equivalences.  In particular, there is a morphism
\[
  Q\cP(\sL,-1)^{\nondeg} \to G\Sp(\fX,-1)
\]
in the homotopy category of simplicial sets.
\end{proposition}

\section{Self-dual quantisations}\label{sdsn}

We now introduce the notion of duality for quantisations, and indicate how it leads to canonical quantisations for line bundles which are Grothendieck--Verdier self-dual, giving rise to the perverse sheaf $\cP\cV$ of vanishing cycles from \cite{BBDJS}. From our point of view, the key property of this sheaf is that it is Verdier self-dual \cite[Equation (2.6)]{BBDJS}, while the object it quantises is Grothendieck--Verdier self-dual.


\subsection{Duality}

We  wish to consider line bundles $\sL$ equipped with an involutive equivalence $\sD(\sL) \simeq \sD(\sL)^{\op}$. 

\begin{example}\label{stratDAex}
 Following Remark \ref{alternateDAremark}, the bar construction for left $\sD_A$-modules is just completion of the \v Cech nerve along the diagonal. This gives an equivalence between left $\sD_X$-modules and quasi-coherent sheaves on the stratified site of $X$ in the sense of \cite{Gr} (equivalently, on the presheaf $X_{\mathrm{strat}}(B):=  \im (X(B) \to X(\H_0B^{\red}))$). Since $\HHom_A(A\hten A,M)=\pr_1^!M$, where $!$ denotes exceptional pullback, it also means that  right $\sD$-modules correspond to $!$-quasi-coherent sheaves. 

Since there is a natural map from the stratified site to the infinitesimal site, this means that left and right crystals in the sense of \cite{GaitsgoryRozenblyumCrystal} give rise to left and right $\sD$-modules in our sense, and in particular that the dualising complex $\omega_X$ on $X$ naturally has the structure of a right $\sD$-module. 

A derived scheme $X$ is said to be Gorenstein when $\omega_X$ is a line bundle,  and if we write $\sE^{\vee}:= \oR\hom_{\sO_X}(\sE,\omega_X)$ for the Grothendieck--Verdier dual of a perfect complex $\sE$, then the right $\sD$-module structure of $\omega_X$ gives a quasi-isomorphism of DGAAs between  $\sD(\sE)^{\op}$ and $\sD(\sE^{\vee})$. 
Thus an equivalence between $ \sD(\sL)$ and $\sD(\sL)^{\op}$  is the same as $\sD(\sL) \simeq \sD(\sL^{\vee}) $, which  is satisfied when $\sL \simeq \sL^{\vee}$  (so $\sL$ is a square root of $\omega_X$). 
\end{example}
 Indeed, an equivalence will exist whenever
 $\sL$ has the structure of a right $\sD(\sL)$-module, or equivalently whenever
$\sL^{\ten 2} $ has the structure of a right $\sD$-module; the equivalence will automatically be involutive as $\sL$ has rank $1$.


\begin{definition}
For a line bundle $\sL$  with a right $\sD$-module structure on $\sL^{\ten 2}$, we  write $(-)^t \co  \sD(\sL)^{\op} \to \sD(\sL)$ for the natural quasi-isomorphism induced by the quasi-isomorphism $\sD^{\op} \simeq \sD(\sL^{\ten 2})$ given by the  right $\sD$-module structure. We then define
\[
 (-)^* \co Q\widehat{\Pol}(\sL,-1) \to Q\widehat{\Pol}(\sL,-1)
\]
by
\[
 \Delta^*(\hbar):= -\Delta^t(-\hbar).
\]
Since this is a quasi-isomorphism of filtered DGLAs, it gives rise to a weak equivalence 
\[
 (-)^* \co Q\widehat{\Pol}(\sL,-1) \to Q\widehat{\Pol}(\sL,-1),
\]
and hence
\[
 Q\cP(\sL,-1) \to Q\cP(\sL,-1)
\]
\end{definition}

The reason for the choice of sign $- \hbar$ in the definition of $\Delta^*$ is that on the associated graded $\gr^F_p \sD_X(\sE) \cong \Symm^p \sT_X$, the operation $(-)^t$ is given  by $(-1)^p$. Thus the underlying Poisson structures satisfy $\pi_{\Delta^*}= \pi_{\Delta}$.  

\begin{definition}\label{selfdualdef}
For a line bundle $\sL$ with $\sL^{\ten 2} $ a right $\sD$-module, the map $(-)^*$ becomes a (homotopy) involution of $Q\widehat{\Pol}(\sL,-1)$, and we define
 $Q\widehat{\Pol}(\sL,-1)^{sd}$ to be the space of homotopy fixed points of the resulting $\Z/2$-action. 

Similarly, we define the space 
\[
 Q\cP(\sL,-1)^{sd}
\]
 of self-dual quantisations to be  the space of  homotopy fixed points of the  $\Z/2$-action  on $Q\cP(\sL,-1)$ generated by $(-)^*$.
\end{definition}

\begin{remark}\label{rightDmodrmk}
Following Remark \ref{DRrmks}, to each  $E_0$ quantisation  $\Delta \in Q\cP(\sL,-1)$ there corresponds a right $\sD_{X}\llbracket \hbar \rrbracket$-module  $\sM_{\hbar}:=(\sL\ten_{\O_{X}}^{\oL}\sD_{X}\llbracket \hbar \rrbracket, \delta +\Delta\cdot\{-\})$. Definition \ref{selfdualdef} says that the quantisation is self-dual with respect to the right $\sD$-module structure on $\sL^{\ten 2} $ when  $\oR\hom_{\sD_X^{\op}\llbracket \hbar \rrbracket}(\sM_{\hbar},\sD_X\llbracket \hbar \rrbracket)\ten_{\O_X}^{\oL} \sL^{\ten 2}$ is equivalent to $\sM_{-\hbar}$ as a right $\sD$-module. This can be phrased as a symmetric perfect pairing
\[
 (\sM_{-\hbar}\ten_{\O_X}^{\oL} \sL^{-1})\ten^{\oL}_{\sD_X(\sL)\llbracket \hbar \rrbracket}( \sM_{\hbar}\ten_{\O_X}^{\oL} \sL^{-1})^{\op} \to \sD(\sL)\llbracket \hbar \rrbracket.
\]

 \end{remark}

\begin{lemma}\label{filtsd}
For the filtration $G$ induced on  $\tilde{F}^pQ\widehat{\Pol}(\sL,-1)^{sd}$ by the corresponding filtration on $\tilde{F}^p Q\widehat{\Pol}(\sL,-1)$, we have
\[
\gr_G^k  \tilde{F}^pQ\widehat{\Pol}(\sL,-1)^{sd} \simeq \begin{cases}
                                                         \gr_G^k  \tilde{F}^pQ\widehat{\Pol}(\sL,-1) & k \text{ even}\\
0 & k \text{ odd}.
                                                        \end{cases}
\]
\end{lemma}
\begin{proof}
As already observed, the involution acts trivially on $\gr_G^0Q\widehat{\Pol}(\sL,-1)$. It therefore acts as multiplication by $(-1)^k$ on   $ \gr_G^kQ\widehat{\Pol}(\sL,-1)= \hbar^k\gr_G^0Q\widehat{\Pol}(\sL,-1)$.
\end{proof}

In particular, this means that $\Ob(Q\cP,\sL,1)^{sd} \simeq 0$, so the map 
\[
 Q\cP(\sL,-1)^{sd}/G^2\to Q\cP(\sL,-1)/G^1\simeq \cP(X,-1)
\]
 is a weak equivalence. In other words, Poisson structures correspond to first order self-dual quantisations. We can say much more in non-degenerate cases:

\begin{proposition}\label{quantpropsd}
For a  line bundle $\sL$  with $\sL^{\ten 2} $ a right $\sD$-module (such as any square root of  $\omega_X$), there is a canonical weak equivalence
\[
  Q\cP(\sL,-1)^{\nondeg,sd} \simeq \cP(X,-1)^{\nondeg} \by \mmc(\hbar^2 \DR(X/R)\llbracket\hbar^2\rrbracket).
\]
In particular, every non-degenerate $(-1)$-shifted Poisson structure gives  a canonical choice of self-dual quantisation of $\sL$.
\end{proposition}
\begin{proof}
Lemma \ref{filtsd} implies that we have weak equivalences
\begin{align*}
 Q\cP(\sL,-1)^{sd}/G^{2i} &\to Q\cP(\sL,-1)^{sd}/G^{2i-1}\\
Q\cP(\sL,-1)^{sd}/G^{2i+1} &\to (Q\cP(\sL,-1)^{sd}/G^{2i})\by^h_{(Q\cP(\sL,-1)/G^{2i})}(Q\cP(\sL,-1)/G^{2i+1}).
\end{align*}

Combined with Propositions \ref{quantprop} and \ref{quantprop2}, these  give  weak equivalences
\[
 Q\cP(\sL,-1)^{\nondeg,sd}/G^{2i+1} \to (Q\cP(\sL,-1)^{\nondeg,sd}/G^{2i})\by \mmc(\hbar^{2i} \DR(X/R))
\]
for all $i>0$, so
\begin{align*}
 Q\cP(\sL,-1)^{\nondeg,sd}/G^{2i+1}&\simeq   (Q\cP(\sL,-1)^{\nondeg,sd}/G^{2i})\by \mmc(\hbar^{2i} \DR(X/R))\\
& \simeq Q\cP(\sL,-1)^{\nondeg}/G^{2i-1}\by \mmc(\hbar^{2i} \DR(X/R)),
\end{align*}
and we have seen that $*$ acts trivially on  $ Q\cP(\sL,-1)/G^1=\cP(\sL,-1) $, so $ Q\cP(\sL,-1)^{sd}/G^1\simeq \cP(\sL,-1)$.
\end{proof}

\begin{remark}\label{oddcoeffsrmk}
The proof of Proposition  \ref{quantpropsd} only shows that for a  self-dual quantisation of a non-degenerate $(-1)$-shifted Poisson structure, the corresponding generalised symplectic structure is determined by its even coefficients. In fact, the odd coefficients must be homotopic to $0$, with the following reasoning. 

As $\mu$ is multiplicative and the de Rham algebra is commutative,  we have a homotopy $\mu(\omega, \Delta)^t \simeq \mu(\omega, -\Delta^t)$ for any $\omega$ and $\Delta$, 
so $\mu(\omega, \Delta)^t(-\hbar)\simeq \mu(\omega, \Delta^*)(\hbar)$. We also have $\sigma(\Delta)^t(-\hbar) \simeq \sigma(\Delta^*)(\hbar)$, so $\omega(\hbar)$ is compatible with $\Delta$ if and only if $\omega(-\hbar)$ is compatible with $\Delta^*$. When $\Delta$ is self-dual and non-degenerate, this implies that  $\omega(\hbar) \simeq \omega(-\hbar)$. 

For a more explicit description of the generalised symplectic structure $\omega$ corresponding to a non-degenerate self-dual quantisation $\Delta$,  observe that we then have an isomorphism
\begin{align*}
 \mu(-, \Delta) \co  &\H^*( F^2\DR(X/R) \by  \hbar^2\DR(X/R)\llbracket\hbar^2\rrbracket) \\
&\to \{v \in \H^*(T_{\Delta}(G*\tilde{F})^2Q\widehat{\Pol}(\sL,-1))~:~ v(-\hbar)= v^t(\hbar)\},
\end{align*}
and that $[\omega]$ must be  the inverse image of $[\hbar^2 \frac{\pd \Delta}{\pd \hbar}]$.
\end{remark}


\subsection{Vanishing cycles}\label{vanishsn}

Given a smooth  scheme $Y$ of dimension $m$ over $\Cx$, and a function $f \co Y \to \bA^1$, we can consider the derived critical locus $X$ of $f$, which is equipped with a canonical $(-1)$-shifted symplectic structure $\omega$. Explicitly, $\sO_X$ is the CDGA given by the alternating algebra $\sO_Y[\sT_Y[1]]$, with differential $\delta$ given by contraction with $df$.

Now, the line bundle $\Omega^m_Y$ pulls back to give a square root $i^*\Omega^m_Y$ of the dualising sheaf $\omega_X$ on $X$. This complex can be written explicitly as $i^*\Omega^m_Y \cong (\Omega_Y^*[m], df\wedge)$. There is a canonical $(-1)$-shifted symplectic structure $\omega$ on $X$, and we write $\Delta_{\omega}$ for the unique compatible self-dual $E_0$ quantisation of Proposition \ref{quantpropsd}. 

\begin{lemma}\label{PVlemma}
 On the derived critical locus $X$, the quantisation $\Delta_{\omega}$ is 
given by $\hbar d \co \Omega_Y^* \to \Omega_Y^*\llbracket\hbar\rrbracket$, for the de Rham differential $d$ on $Y$.
\end{lemma}
\begin{proof}
We first need to check that $\Delta:=\Delta_{\omega}$ is self-dual, but this follows because we have a pairing
\[
 (i^*\Omega^m_Y\llbracket\hbar\rrbracket, \delta + \hbar d)\ten_R (i^*\Omega^m_Y\llbracket\hbar\rrbracket, \delta - \hbar d) \to (i^*\Omega^m_Y)^{\ten 2}\llbracket\hbar\rrbracket= \omega_X\llbracket\hbar\rrbracket
\]
given by combining the cup product  $\Omega^*_Y[m]\ten\Omega^*_Y[m] \to \Omega^*_Y[2m]$ with projection to $\Omega^m_Y[m] $ followed by inclusion in $\omega_X \cong  \Omega^m_Y\ten_{\sO_Y} \Omega^*_Y[m]$.

Now, for $y \in \sO_Y \subset \sO_X$, the differential operator $[\Delta, y]$ acts on $\Omega^*_Y[m] $ as multiplication by $\hbar
dy$. For $\eta \in \sT_Y \subset \sO_X$, the differential operator $[\Delta, \eta]$ acts on $\Omega^*_Y[m] $ as the Lie derivative $\hbar\Lie_{\eta}$. Thus the ring homomorphism $\mu(-, \Delta)$ is given on generators of $\Omega^1_X$ by
\[
 \mu(dy,\Delta)= \hbar dy\wedge\,, \quad \mu(d\eta,\Delta) = \hbar\Lie_{\eta},
\]
 and we need to show that it maps  $\omega$ to $\hbar^2\frac{\pd \Delta}{\pd \hbar}= \hbar^2 d$, up to homotopy. 

There is a canonical representative  $\omega \in \z^{-1}\Gamma(X, \Omega^2_X)$ with $d\omega=0$, and on any affine open we may lift it to an element of   $\tilde{\omega} \in \z^1F^2\DR'(X/R)$ lying in  $N^2\hat{\O}_X^{\ten \bt +1}$. Any such choices differ by $d\alpha$ for $\alpha \in F^2N^1\hat{\O}_X^{\ten \bt +1}$; in other words, $\alpha \in I^2$, for $I= \ker(\hat{\O}_X^{\ten 2} \to \O_X)$. We now just observe that $\mu(\alpha, \Delta)$ is anti-self-dual, since $\mu(da.db,\Delta)= \hbar \<a,b\>_{\pi} \in \hbar \O_X$. Therefore  $\mu(\tilde{\omega},\Delta)+ \mu(\tilde{\omega},\Delta)^*$ is strictly independent of the choice of lift. 

It thus suffices to show that for some such local choice we have $\mu(\tilde{\omega},\Delta)=\sigma(\Delta)$.
 To make things explicit, we now take local co-ordinates $y_1, \ldots, y_m$, and write $\eta_i \in \sO_X$ for the element given by $\pd_{y_i} \in \sT_Y$,   so $X$ has co-ordinates $y_1, \ldots, y_m, \eta_1, \ldots, \eta_m$. The generator $dy_1 \wedge \ldots \wedge dy_m$ then gives an isomorphism $i^*\Omega^m_Y \cong \sO_X$, and $\Delta$ corresponds to the quantisation of $\sO_X$ given by
$\Delta :=\hbar\sum_i  \pd_{y_i}\pd_{\eta_i}$.
 
The calculations above reduce to 
\begin{align*}
 \mu(y_i\ten 1 - 1\ten y_i, \Delta)= \mu(dy_i, \Delta) &= \hbar \pd_{\eta_i}\\
\mu(\eta_i\ten 1 - 1\ten \eta_i, \Delta)= \mu(d\eta_i, \Delta) &= \hbar \pd_{y_i}.
\end{align*}

We now choose the lift $\tilde{\omega}:=\sum_i dy_i \smile d\eta_i \in \z^1F^2\DR'(X/R)$ of the  canonical $(-1)$-shifted symplectic structure $\omega= \sum_i dy_i \wedge d\eta_i \in \z^1F^2\DR(X/R)$.
Since $\mu(-, \Delta)$ is multiplicative, it follows that
\[
 \mu(dy_i \smile d\eta_i, \Delta) =  \hbar^2\pd_{\eta_i}\pd_{y_i},
\]
so 
\[
\mu(\tilde{\omega}, \Delta) = \mu(\sum_i dy_i \smile d\eta_i, \Delta)= \sum_i \hbar^2 \pd_{\eta_i}\pd_{y_i}=  \hbar^2\frac{\pd \Delta}{\pd \hbar}= \sigma(\Delta).
\]
 \end{proof}

\begin{proposition}\label{PVprop}
On the derived critical locus $(X, \omega)$, the localisation 
\[
 \oR\Gamma(X, (M\llbracket \hbar\rrbracket, \delta + \Delta_{\omega})\ten_{\Cx[\hbar]}\Cx[\hbar, \hbar^{-1}] 
\]
is quasi-isomorphic to
\[
 \oR\Gamma(X, \cP\cV^{\bt}_{Y,f})((\hbar)),
\]
for the perverse sheaf $\cP\cV$ of vanishing cycles from \cite[\S 2.4]{BBDJS}.
\end{proposition}
\begin{proof}
By \cite[Theorem 1.1]{SabbahTwistedII}, 
\[
 \oR\Gamma(X, \cP\cV^{\bt}_{Y,f})((u)) \simeq (\Omega^*_Y((u)), d - u^{-1}df)[m].
\]
Multiplying by $\hbar^i$ in degree $i$, and setting $\hbar=-u$, we get
\[
 \oR\Gamma(X, \cP\cV^{\bt}_{Y,f})((\hbar)) \simeq (\Omega^*_Y((\hbar)), df + \hbar d)[m];
\]
 Lemma \ref{PVlemma} completes the proof.
\end{proof}

\subsection{Quantisation for $n$-shifted symplectic structures}\label{nonnegsn}

We now discuss how to generalise these results to more general $n$-shifted structures, including the non-trivial cases $n=0,-2$. 

\subsubsection{Unshifted Poisson structures}
To address the case $n=0$, replace the filtered DGAA $\sD_{A}$ of differential operators with the filtered DGLA  $\sD_{A}^{\mathrm{poly}}[1]$ of polydifferential operators, setting 
$
 Q\widehat{\Pol}(A,0):= \prod_{p \ge 0}F_p\sD_{A}^{\mathrm{poly}}\hbar^{p-1}.
$
As in \cite{vdberghGlobalDQ}, the HKR isomorphism leads to a quasi-isomorphism between $\sD_{A}^{\mathrm{poly}}$ and the Hochschild complex of $A$ over $R$. 

For a quantisation 
\[
 \Delta \in Q\cP(A,0):= \mmc(Q\widehat{\Pol}(A,0)[1]),
\]
 the centre $T_{\Delta}Q\widehat{\Pol}(A,0):= (\hbar Q\widehat{\Pol}(A,0), \delta + [\Delta,-])$ then has the canonical structure of an $E_2$-algebra.
A choice of formality isomorphism for $E_2$ will therefore give a $P_2$-algebra structure on $T_{\Delta}Q\widehat{\Pol}(A,0)$, and we may define a CDGA map 
\[
 \mu(-,\Delta) \co \DR(A) \to T_{\Delta}Q\widehat{\Pol}(A,0)
\]
determined on generators by $\mu(a, \Delta)=a$, $\mu(df, \Delta)= [\Delta, f]$ for $a,f \in A$. This will be a quasi-isomorphism when $\Delta$ is non-degenerate.

Such a construction  yields  analogues of all the results in \S\S \ref{affinesn}, \ref{DMsn}, with a cohomological shift. In particular,
there is a map from quantisations to $\prod_{k\ge 0}\H^{2}( F^{2-2k}\DR(A)) \hbar^{k}$, and 
  the analogue of Propositions \ref{quantprop}, \ref{quantprop2}  says that a deformation quantisation of a non-degenerate Poisson structure exists  whenever it can be quantised to first order. The analogue of self-dual $E_0$-quantisations are DQ algebroid quantisations $\sB_{\hbar}$ equipped with involutions  $\sB_{-\hbar} \simeq \sB_{\hbar}^{\op}$ --- for details, see \cite{DQnonneg}.


\subsubsection{Positively shifted Poisson structures}
For $n\ge 1$, we can likewise define $Q\widehat{\Pol}(A,n) $ in terms of shifted differential operators or $E_{n+1}$-Hochschild complexes over $\Rees(\DR(A))$. However, formality of $E_{n+1}$ should yield  equivalences $Q\widehat{\Pol}(A,n) \simeq \widehat{\Pol}(A,n)\llbracket \hbar \rrbracket$ and $Q\cP(A/R,n) \simeq Q\cP(A\llbracket \hbar\rrbracket/R\llbracket \hbar \rrbracket,n)$, making the analogues of Propositions \ref{QcompatP1} and \ref{quantprop} less interesting.

(We might also wish to quantise higher analogues of line bundles: for  derived  stacks, these should be classes in $\H^{n+2}(\fX, \bG_m)$.  For unbounded CDGAs $A$, the analogue of the strict line bundle $A_b$ of Definition \ref{strictlb} is the curved $E_{n+1}$-algebra $(A,c)$ for $c \in \z^{n+2}(A)$ --- because $A$ is an $E_{\infty}$ algebra, its Lie bracket is trivial,  so we still have $\delta^2= [c,-]$.)

\subsubsection{$(-2)$-shifted Poisson structures}

For $n \ge -1$, there is a canonical $E_{n+2}$-algebra quantisation  $\hbar Q\Pol(A,n)$ of $\hbar \Pol(A,n)$ given by  $(n+1)$-shifted differential operators, which is equipped with an $E_{n+2}$-algebra morphism $A \to \hbar Q\Pol(A,n)$.
In order to adapt the techniques of this paper to $(-2)$-shifted symplectic structures, \cite{DQ-2} starts with the data of an $E_0$ quantisation 
\[
 \hbar Q\widehat{\Pol}(A,-2):= (\hbar\widehat{\Pol}(A,-2)\llbracket \hbar \rrbracket, \delta + \Delta)
\]
 of the $P_0$-algebra $\hbar \widehat{\Pol}(A,-2)$, with $\Delta \in Q\cP(\hbar\Pol(A,-2),-1)$ satisfying $\Delta(A)=0$ (but not necessarily $A$-linear) and lifting the 
canonical Poisson bracket on $\hbar \widehat{\Pol}(A,-2)$. 

In particular, the condition $\Delta(A)=0$ implies $\Delta(1)=0$, making $\hbar Q\widehat{\Pol}(A,-2)$   a $BV_{\infty}$-algebra. It thus suffices to have a right $\sD$-module structure on $A$,  and to define $\hbar Q\widehat{\Pol}(A,-2)$ in terms of the right de Rham complex of $A$, as in \cite{schechtmanRksBV}. However, $(-2)$-shifted symplectic derived schemes are seldom Gorenstein, so a right $\sD$-module structure on $A$ is not necessarily equivalent to a left $\sD$-module structure on the dualising complex $\omega_A$. 

 Given a $(-2)$-shifted Poisson structure $\pi \in \cP(A,-2)$, \cite{DQ-2}  then defines an $E_{-1}$ quantisation of $\pi$ to be a lift of $\hbar\pi$ to an element $\hbar S$ of the $BV_{\infty}$-algebra $\hbar \tilde{F}^2Q\widehat{\Pol}(A,-2)$ satisfying the $L_{\infty}$ Maurer--Cartan equation, or equivalently the quantum master equation $(\delta + \Delta) e^{S}=0$. When there exists a morphism $\DR^r(A) \to \DR^r(\omega_A)[-v\dim A]$ of right de Rham cohomology complexes, this  gives rise to an element of degree $v\dim A$ in Borel--Moore homology, permitting comparison with \cite{BorisovJoyce}. 

 Writing $\Delta_S := \ad_{(e^{-S})}\Delta + \delta S$, 
the centre $T_SQ\widehat{\Pol}(A,-2)$ of $S$ is defined to be the complex $(\hbar\widehat{\Pol}(A,-2)\llbracket \hbar \rrbracket,\delta + \Delta_S)$. We can then define a  compatibility map $\mu$ by composing the map 
\[
 \DR'(A) \to \DR'(\Pol(A,-2)) \xra{ \mu(-, \Delta_S)}  T_{\Delta_S}Q\widehat{\Pol}(\Pol(A,-2),-1)
\]
from this paper with the evaluation of differential operators at $1$. Since $\Delta_S(1)=0$ by the quantum master equation, this gives a filtered map
\[
 \mu(-, S)\co \DR'(A) \to T_SQ\widehat{\Pol}(A,-2),
\]
and a $(-2)$-shifted pre-symplectic structure $\omega$ is then said to be compatible with $S$ when $\mu(\omega, S) \simeq \hbar^2  \frac{\pd S}{\pd \hbar}$. 

To first order, we have $\Delta_S\cong \Delta + \{\pi,-\}$; since $\Delta(A)=0$, the map  $\mu(-, S) \co \DR'(A) \to T_SQ\widehat{\Pol}(A,-2)/G^1 = T_{\pi}\widehat{\Pol}(A,-2)$ is then just the compatibility map $\mu(-, \pi)$ from \cite{poisson}.

\subsubsection{Self-duality}
In \S \ref{sdsn}, the key to self-duality for $E_0$ quantisations of $\sqrt \omega_X$  is the filtered  involution $(-)^t \co \sD_X(\sqrt \omega_X)\simeq \sD_X(\sqrt \omega_X)^{\op}$. For $n\ge 0$,  HKR isomorphisms mean that the analogue of $\sD_A$ is the higher Hochschild cohomology complex $\HH^{E_{n+1}}(A)$ with its $E_{n+2}$-algebra structure. In order to define self-dual quantisations, we would thus need a filtered  involution $(-)^t\co \HH^{E_{n+1}}(A) \simeq \HH^{E_{n+1}}(A)^{\op}$, lifting the $P_{n+2}$-algebra involution of $\widehat{\Pol}(X, n)$ given by $(-1)^m$ on $m$-vectors. 

Of course, when $n >0$,  the equivalence $Q\widehat{\Pol}(A,n) \simeq \widehat{\Pol}(A,n)\llbracket \hbar \rrbracket$ coming from formality of $E_{n+1}$  allows us to transfer the involution on $\widehat{\Pol}(A,n) $ to $Q\widehat{\Pol}(A,n)$. We then have
$Q\widehat{\Pol}(A,n)^{sd} \simeq \widehat{\Pol}(A,n)\llbracket \hbar^2 \rrbracket$, giving  a sense in which the canonical quantisations coming from formality of $E_{n+1}$ are self-dual. 

Whereas Grothendieck--Verdier self-duality for a line bundle $\sL$ is an involutive equivalence $\sL \simeq \oR\cHom_{\O_X}(-, \omega_X)$, the obvious notion of self duality for an algebroid $\cA$ is an involution $\cA \simeq \cA^{\op}$. 
When $n=0$,  an involutive  filtered $E_2$-equivalence  on the Hochschild complex of $X$  gives an analogue of Proposition \ref{quantpropsd}, generating self-dual quantisations from symplectic structures. This amounts to looking for     DQ algebroid quantisations $\cA$ equipped with involutions $\cA(-\hbar) \simeq \cA(\hbar)^{\op}$ deforming a chosen contravariant involution on the Picard algebroid (or even any $2$-line bundle). Such  involutions correspond to line bundles, the obvious choices being  $\oR\cHom_{\O_X}(-, \O_X)$ and $\oR\cHom_{\O_X}(-, \omega_X)$, and involutive $2$-line bundles are $\mu_2$-gerbes. For more details, see \cite{DQnonneg}.


\bibliographystyle{alphanum}
\bibliography{references.bib}

\newcommand{\etalchar}[1]{$^{#1}$}
\def\cprime{$'$}
\begin{thebibliography}{BBD{\etalchar{+}}}

\bibitem[BBD{\etalchar{+}}]{BBDJS}
C.~Brav, V.~Bussi, D.~Dupont, D.~Joyce, and B.~Szendr{\"o}i.
\newblock Symmetries and stabilization for sheaves of vanishing cycles.
\newblock {\em J. Singul.}, 11:85--151, 2015.
\newblock arXiv:1211.3259 [math.AG]. With an appendix by J{\"o}rg
  Sch{\"u}rmann.

\bibitem[Beh1]{behrendDTmicrolocal}
Kai Behrend.
\newblock Donaldson-{T}homas type invariants via microlocal geometry.
\newblock {\em Ann. of Math. (2)}, 170(3):1307--1338, 2009.

\bibitem[Beh2]{behrendICM}
Kai Behrend.
\newblock On the virtual fundamental class.
\newblock In {\em Proceedings of the International Congress of Mathematicians
  (Seoul 2014), Vol. II}, pages 591--614, 2014.
\newblock http://www.icm2014.org/download/Proceedings\_Volume\_II.pdf.

\bibitem[BJ]{BorisovJoyce}
D.~{Borisov} and D.~{Joyce}.
\newblock {Virtual fundamental classes for moduli spaces of sheaves on
  Calabi-Yau four-folds}.
\newblock {\em Geom. Topol.}, to appear. 2015.
\newblock arXiv: 1504.00690 [math.AG].

\bibitem[GR]{GaitsgoryRozenblyumCrystal}
Dennis Gaitsgory and Nick Rozenblyum.
\newblock Notes on geometric {L}anglands: crystals and {$D$}-modules.
\newblock arXiv:1111.2087, 2011.

\bibitem[Gro]{Gr}
A.~Grothendieck.
\newblock Crystals and the de {R}ham cohomology of schemes.
\newblock In {\em Dix Expos{\'e}s sur la Cohomologie des Sch{\'e}mas}, pages
  306--358. North-Holland, Amsterdam, 1968.

\bibitem[Hin]{hinstack}
Vladimir Hinich.
\newblock D{G} coalgebras as formal stacks.
\newblock {\em J. Pure Appl. Algebra}, 162(2-3):209--250, 2001.

\bibitem[Kra]{kravchenko}
Olga Kravchenko.
\newblock Deformations of {B}atalin--{V}ilkovisky algebras.
\newblock In {\em Poisson geometry (Warsaw, 1998)}, volume~51 of {\em Banach
  Center Publ.}, pages 131--139, Warsaw, 2000. Polish Acad. Sci.

\bibitem[Pos]{positselskiDerivedCategories}
Leonid Positselski.
\newblock Two kinds of derived categories, {K}oszul duality, and
  comodule-contramodule correspondence.
\newblock {\em Mem. Amer. Math. Soc.}, 212(996):vi+133, 2011.

\bibitem[Pri1]{ddt1}
J.~P. Pridham.
\newblock Unifying derived deformation theories.
\newblock {\em Adv. Math.}, 224(3):772--826, 2010.
\newblock arXiv:0705.0344v6 [math.AG], corrigendum 228 (2011), no. 4,
  2554--2556.

\bibitem[Pri2]{dmc}
J.~P. Pridham.
\newblock Constructing derived moduli stacks.
\newblock {\em Geom. Topol.}, 17(3):1417--1495, 2013.
\newblock arXiv:1101.3300v2 [math.AG].

\bibitem[Pri3]{stacks2}
J.~P. Pridham.
\newblock Presenting higher stacks as simplicial schemes.
\newblock {\em Adv. Math.}, 238:184--245, 2013.
\newblock arXiv:0905.4044v4 [math.AG].

\bibitem[Pri4]{poisson}
J.~P. Pridham.
\newblock Shifted {P}oisson and symplectic structures on derived {$N$}-stacks.
\newblock {\em J. Topol.}, 10(1):178--210, 2017.
\newblock arXiv:1504.01940v5 [math.AG].

\bibitem[Pri5]{DQ-2}
J.~P. Pridham.
\newblock Deformation quantisation for {$(-2)$}-shifted symplectic structures.
\newblock arXiv: 1809.11028v1 [math.AG], 2018.

\bibitem[Pri6]{DQnonneg}
J.~P. Pridham.
\newblock Deformation quantisation for unshifted symplectic structures on
  derived {A}rtin stacks.
\newblock {\em Selecta Math. (N.S.)}, 24(4):3027--3059, 2018.
\newblock arXiv: 1604.04458v4 [math.AG].

\bibitem[PTVV]{PTVV}
T.~Pantev, B.~To{\"e}n, M.~Vaqui{\'e}, and G.~Vezzosi.
\newblock Shifted symplectic structures.
\newblock {\em Publ. Math. Inst. Hautes \'Etudes Sci.}, 117:271--328, 2013.
\newblock arXiv: 1111.3209v4 [math.AG].

\bibitem[Sab]{SabbahTwistedII}
Claude Sabbah.
\newblock On a twisted de {R}ham complex, {II}.
\newblock arXiv:1012.3818v1 [math.AG], 2010.

\bibitem[Sch]{schechtmanRksBV}
Vadim Schechtman.
\newblock Remarks on formal deformations and {B}atalin--{V}ilkovisky algebras.
\newblock arXiv:math/9802006, 1998.

\bibitem[To{\"e}]{toenICM}
Bertrand To{\"e}n.
\newblock Derived algebraic geometry and deformation quantization.
\newblock In {\em Proceedings of the {I}nternational {C}ongress of
  {M}athematicians (Seoul 2014), Vol. II}, pages 769--752, 2014.
\newblock arXiv:1403.6995v4 [math.AG].

\bibitem[TV]{hag2}
Bertrand To{\"e}n and Gabriele Vezzosi.
\newblock Homotopical algebraic geometry. {II}. {G}eometric stacks and
  applications.
\newblock {\em Mem. Amer. Math. Soc.}, 193(902):x+224, 2008.
\newblock arXiv math.AG/0404373 v7.

\bibitem[VdB]{vdberghGlobalDQ}
Michel Van~den Bergh.
\newblock On global deformation quantization in the algebraic case.
\newblock {\em J. Algebra}, 315(1):326--395, 2007.

\end{thebibliography}
\end{document}